\documentclass[a4paper,12pt]{amsart}
\usepackage{vmargin}
\usepackage[pdftex]{graphicx}
\usepackage{amsfonts}
\usepackage{amssymb}
\usepackage{mathrsfs}
\usepackage{stmaryrd}
\usepackage{amsmath}
\usepackage{amsthm}
\usepackage{enumerate}
\usepackage{nomencl}
\usepackage[bookmarks=true]{hyperref}   
\usepackage{esint}
\usepackage{dsfont}
\usepackage{upgreek}
\usepackage{color}
\usepackage{comment}
\usepackage{float}
\usepackage[utf8]{inputenc}

\makeatletter
\renewcommand\section{\@startsection{section}{1}{\z@}%
                       {-3\p@ \@plus -4\p@ \@minus -4\p@}%
                       {3\p@ \@plus 4\p@ \@minus 4\p@}%
                      {\normalfont\normalsize\centering\scshape}}
\makeatother

\hypersetup{
    colorlinks,%
}

\author{Lashi Bandara}
\author{Sajjad Lakzian}
\author{Michael Munn}
\title{Geometric singularities 
and a flow tangent to the Ricci flow}
\date{\today}
\address{Lashi Bandara, Mathematical Sciences,
Chalmers University of Technology and University of Gothenburg, SE-412 96, Gothenburg, Sweden}
\urladdr{\href{http://www.math.chalmers.se/~lashitha}{http://www.math.chalmers.se/~lashitha}}
\email{\href{mailto:lashi.bandara@chalmers.se}{lashi.bandara@chalmers.se}}

\address{Sajjad Lakzian,
Hausdorff Center for Mathematics, 
D-53115, Bonn, Germany } 
\urladdr{\href{http://wt.iam.uni-bonn.de/sajjad-lakzian}{http://wt.iam.uni-bonn.de/sajjad-lakzian}}
\email{\href{mailto:lakzians@gmail.com}{lakzians@gmail.com}}

\address{Michael Munn,
Department of Mathematics, 
Courant Institute of Mathematical Sciences, New York, NY 10012-1185, USA}
\urladdr{\href{http://math.nyu.edu/~mmunn/}{http://math.nyu.edu/~mmunn}}
\email{\href{mailto:munn@nyu.edu}{munn@nyu.edu}}


\def\colour{\colour}

\def\colour{\color}

\newtheorem{theorem}{Theorem}[section]
\newtheorem{corollary}[theorem]{Corollary}
\newtheorem{lemma}[theorem]{Lemma}
\newtheorem{proposition}[theorem]{Proposition}

\newtheorem{definition}[theorem]{Definition}
\newtheorem{remark}[theorem]{Remark}

\newcommand{\mdot}{\cdotp}

\newcommand{\cbrac}[1]{\left(#1\right)}
\newcommand{\bbrac}[1]{\left[#1\right]}
\newcommand{\dbrac}[1]{\left\{#1\right\}}
\newcommand{\modulus}[1]{|#1|}
\newcommand{\lmodulus}[1]{\left|#1\right|}
\newcommand{\set}[1]{\dbrac{#1}}

\newcommand{\dom}{ {\mathcal{D}}}
\newcommand{\ran}{ {\mathcal{R}}}
\newcommand{\nul}{ {\mathcal{N}}}

\newcommand{\comp}{\circ}

\newcommand{\e}{\mathrm{e}}

\newcommand{\R}{\mathbb{R}}
\newcommand{\C}{\mathbb{C}}

\newcommand{\script}[1]{\mathscr{#1}}

\renewcommand{\emptyset}{\varnothing}
\newcommand{\union}{\cup}

\newcommand{\intersect}{\cap}

\newcommand{\rest}[1]{{{\lvert_{}}_{}}_{#1}}
\newcommand{\close}[1]{\overline{#1}}		

\newcommand{\ind}[1]{\raisebox{\depth}{\(\chi\)}_{#1}}	

\renewcommand{\epsilon}{\varepsilon}

\renewcommand{\phi}{\varphi}

\newcommand{\tp}[1]{{#1}^{\mathrm{tr}}}

\newcommand{\tensor}{\otimes}

\newcommand{\norm}[1]{\| #1 \|}			
\DeclareMathOperator{\esssup}{esssup}

\DeclareMathOperator{\len}{\ell}			

\DeclareMathOperator{\divv}{div}		

\newcommand{\ddt}[1][t]{\frac{d}{d#1}}		

\newcommand{\Ric}{{\rm Ric}}			

\newcommand{\bnd}{\partial}			
\newcommand{\Forms}[1][{}]{\mathbf{\Omega}^{#1}}		
\newcommand{\Tensors}[1][{}]{{\mathcal{T}}^{(#1)}}	

\newcommand{\Sect}{\mathbf{\Gamma}}		
\newcommand{\tanb}{{\rm T}}		
\newcommand{\cotanb}{{\rm T}^\ast}	

\newcommand{\pullb}[1]{{#1}^\ast}			

\DeclareFontFamily{OT1}{restrictfont}{}
\DeclareFontShape{OT1}{restrictfont}{m}{n}{<-> fmvr8x}{}

\newcommand{\extd}{{\rm d}}			

\newcommand{\Dir}{{\rm D}}			

\newcommand{\inprod}[1]{\langle #1 \rangle}	
\newcommand{\grad}{\nabla}			
\newcommand{\conn}[1][{}]{{\grad_{{#1}}}}		
\DeclareMathOperator{\Lip}{\bf Lip}			
\newcommand{\Leb}[1][{}]{\script{L}^{#1}}			

\DeclareMathOperator{\diag}{diag}
\newcommand{\bddlf}{\mathcal{L}} 	
\newcommand{\spec}{\sigma}		
\newcommand{\rset}{\rho}				
\DeclareMathOperator{\nr}{nr}				

\newcommand{\Lp}[2][{}]{{\rm L}^{#2}_{\rm #1}}		
\newcommand{\Ck}[2][{}]{{\rm C}^{#2}_{\rm #1}}		
\newcommand{\Sob}[2][{}]{{\rm W}^{#2}_{\rm #1}}		

\newcommand{\ddelta}{\updelta}

\newcommand{\iden}{{\mathrm{I}}}

\newcommand{\Lap}{\Delta}			

\newcommand{\p}{p}				

\newcommand{\sP}{\script{P}} 

\newcommand{\cC}{\mathcal{C}}

\newcommand{\cV}{\mathcal{V}}
\newcommand{\cM}{\mathcal{M}} 
\newcommand{\cN}{\mathcal{N}}

\newcommand{\cS}{\mathcal{S}}

\newcommand{\Spa}{\mathcal{X}}

\newcommand{\mg}{\mathrm{g}}
\newcommand{\mgt}{{\tilde{\mg}}}
\newcommand{\mh}{\mathrm{h}}

\newcommand{\met}{\mathrm{d}} 
\newcommand{\Sph}{\mathrm{S}}

\newcommand{\Poincare}{Poincar\'e~}		

\newcommand{\Div}{\mathrm{L}}

\newcommand{\B}{\mathrm{B}}

\newcommand{\hk}{\uprho}

\DeclareMathOperator{\RCD}{RCD}
\DeclareMathOperator{\CD}{CD}
\DeclareMathOperator{\witch}{witch}

\newcommand{\RNum}[1]{\uppercase\expandafter{\romannumeral #1\relax}}

\begin{document}

\maketitle

\begin{abstract}
 We consider a geometric flow introduced by Gigli and Mantegazza
which, in the case of a smooth compact manifold
with a smooth metric,
is tangential to the Ricci flow
almost-everywhere along geodesics.
To study spaces with geometric
singularities, we consider this flow in the context 
of a smooth manifold with a rough metric possessing 
a sufficiently regular heat kernel. On an appropriate 
non-singular open region, we provide
a family of metric tensors evolving in time 
and provide a regularity theory
for this flow in terms of the regularity of
the heat kernel. 

 When the rough 
metric induces a metric measure space  satisfying  a Riemannian Curvature Dimension condition, 
we demonstrate that the distance induced
by the flow is identical to the
evolving distance metric 
defined by Gigli and Mantegazza
on appropriate admissible points.
Consequently, we demonstrate that 
a smooth compact manifold with a finite number of
geometric conical singularities remains
a smooth manifold with a smooth metric away 
from the cone points for all future times.
Moreover, we show that the distance induced by the
evolving metric tensor agrees with the flow of $\RCD(K,N)$ spaces defined by Gigli-Mantegazza.
\end{abstract}
\vspace*{-0.5em}
\tableofcontents
\vspace*{-2em}

\parindent0cm
\setlength{\parskip}{\baselineskip}

\section{Introduction}

Nearly ten years ago, using the tools of optimal transportation,  Lott-Villani \cite{LottVillani} and Sturm \cite{SturmI, SturmII} extended the notion of lower Ricci curvature bounds to the setting of general metric measure spaces. Among other things, they showed that this so-called curvature-dimension condition, denoted $\CD(K,N)$ for $K\in \mathbb{R}$, $N\in [1, \infty]$, is stable under measured Gromov-Hausdorff limits and consistent with the notion of Ricci curvature lower bounds for Riemannian manifolds.  That is to say, for smooth Riemannian manifolds, the $\CD(K,N)$ condition is equivalent to having the Ricci curvature tensor bounded below by $K$ and dimension of the manifold at most $N$. In a similar way, for a metric measure space $(\Spa, \met, \mu)$, the  $\CD(K,N)$ condition is understood to say that $\Spa$ has $N$-dimensional Ricci curvature bounded below by $K$. 

Although $\CD(K,N)$ spaces enjoy many favourable properties, Villani shows in \cite{Villani} that such spaces also allow for Finsler structures. This is a somewhat unsettling fact as it is known that Finsler manifolds cannot arise as the Gromov-Hausdorff limit of Riemannian manifolds with lower Ricci curvature bounds. Even more so, classical results such as the Cheeger-Gromoll splitting theorem were known to fail for general metric measure spaces which are merely $\CD(K,N)$. In order to retain these nice properties while also ruling out Finsler geometries, Ambrosio-Gigli-Savar\'{e} introduced a further refined version of the curvature-dimension bound requiring that in addition, the Sobolev space $\Sob{1,2}(\Spa)$ is a Hilbert space. Combining  this condition of \emph{infinitesimally Hilbertian} structure with the classical curvature dimension condition, they define the \emph{Riemannian Curvature Dimension} condition, denoted $\RCD(K,N)$.

In recent years there has been an increased interest in better understanding the fine geometric and analytic consequences of this Riemannian curvature dimension condition. There has been a  great deal of progress in this direction and a number of very deep results describing the structure of these spaces. 
See, for example, recent work of Ambrosio, Cavalletti, Gigli, Mondino, Naber, Rajala, Savar\'e, Sturm in \cite{AGMR, AGS-BE, Cav, GM, Gigli-split,GMR, MN}.
We emphasise that this list is by no means exhaustive, and encourage the reader to consult the references within.

The starting point of our
considerations is the 
paper \cite{GM} by 
Gigli and Mantegazza, 
where they define a geometric flow
for spaces that are possibly singular.
There, the authors 
consider a compact $\RCD(K,N)$ space
$(\Spa, \met, \mu)$
and define a family of evolving
distance metrics $\met_t$ for positive time.
They build this via the heat flow of $\met$ and $\mu$
in \emph{Wasserstein space}, 
the space of probability measures on $\Spa$
with the so-called Wasserstein metric.
The essential feature
of this flow is when the triple
$(\Spa, \met,\mu)$ arises
from a smooth compact manifold $(\cM,\mg)$.
In this setting, the evolution $\met_t$
is given by an evolving smooth metric tensor
which satisfies
$$ \partial_t \mg_t(\dot{\gamma}(s), \dot{\gamma}(s))\rest{t = 0} = -2 \Ric_{\mg}(\dot{\gamma}(s), \dot{\gamma}(s)),$$
for almost-every $s \in [0,1]$
along $\mg$-geodesics $\gamma$.
That is, $\mg_t$ is
\emph{tangential} to the Ricci flow 
in this weak sense.  This work of Gigli-Mantegazza gives one direction in 
which one could  possibly  define a Ricci flow for 
general metric measure spaces for all $t >0$.

The latter correspondence is
obtained by writing an evolving metric
tensor via a partial differential equation.
First, at each $t > 0$, $x \in \cM$
and $v \in \tanb_x \cM$, 
they consider 
the \emph{continuity equation}
\begin{equation}
\tag{CE}
\label{Def:E}
\begin{aligned}
&-\divv_\mg ( \hk^\mg_t(x,y) \conn \phi_{t,x,v}(y)) =( \extd_x\hk^\mg_t(x,y))(v) \\
&\int_{\cM} \phi_{t,x,v}(y)\ d\mu_\mg(y) = 0.
\end{aligned} 
\end{equation}
For smooth metrics, 
the existence and regularity of solutions
to this flow are immediate and 
they define a smooth family 
of metrics evolving in time by
\begin{equation} 
\tag{GM}
\label{Def:GM}
\mg_t(u,v)(x) = \int_{\cM} \mg(\conn \phi_{t,x,u}(y), 
	\conn \phi_{t,x,v}(y))\ \hk^\mg_t(x,y)\ d\mu_\mg(y).
\end{equation}

While the formulation of this flow for $\RCD(K,N)$ spaces gives the existence of a 
time evolving distance metric
for a possible singular space, 
it reveals little regularity information
in positive time.
On the other hand, the evolving metric tensor
is only specified when both the initial metric
and underlying manifold are smooth. 

One of the motivating questions
of the current paper is to better understand
the behaviour of this flow on manifolds
with geometric singularities and hence, the question of regularity
will be a primary focus.
Since there are few 
tools in the setting of $\RCD(K,N)$ spaces that are
sufficiently mature to extract regularity 
information, we restrict ourselves exclusively 
to compact manifolds that are \emph{smooth}, by which
we assume only that they admit a smooth differential structure.
While this may seem a potentially severe restriction,
we vindicate ourselves by allowing for
the metric tensor to be \emph{rough}, i.e.,
a symmetric, positive-definite, $(2,0)$-tensor field
with measurable coefficients.
Rough metrics and their salient features
are discussed in \S\ref{Sec:Rough}.

Such metrics allow a wide class 
of phenomena, so large that
such a metric may not even
induce a length structure, only 
an $n$-dimensional measure. 
Moreover, they may induce
spaces that are not $\RCD$.
However, this potentially outrageous
behaviour is redeemed 
by the fact that they are able to capture a
wide class of \emph{geometric singularities}, including 
Lipschitz transformations of $\Ck{1}$
geometries, conical singularities, and Euclidean 
boxes. These objects
are considered in \S\ref{Sec:Sing}.

Our primary concern is
when a metric exhibits singular 
behaviour on some closed subset $\cS \neq \cM$, 
but has good regularity properties on the
open set $\cM \setminus \cS$. This is indeed
the case for a Euclidean box, or a
smooth compact manifold with 
finite number of geometric conical singularities.
In this situation, away  from the singular part, we are able
to construct a metric tensor $\mg_t$  evolving according
to \eqref{Def:GM}.
We say that two points $x, y \in \cM \setminus \cS$
are $\mg_t$-admissible  if 
for any absolutely continuous curve $\gamma:I \to \cM$
connecting these points,  
there is another absolutely continuous curve $\gamma':I \to\cM$
between $x$ and $y$
with length (measured via $\met_t$)
less than $\gamma$
and for which $\gamma'(s) \in \cM \setminus \cS$
for almost-every $s$. For such a pair of points,
we assert that the distance $\met_t(x,y)$,
given by the $\RCD(K,N)$-flow
of Gigli and Mantegazza, is induced by the metric tensor
$\mg_t$.
The following is a more precise showcasing of our 
main theorem. It is proved in 
\S\ref{Sec:RCD}. 

\begin{theorem}
\label{Thm:Main}
Let $\cM$ be a smooth, compact manifold
with rough metric $\mg$ 
that induces a distance metric $\met_\mg$.
Moreover, suppose there exists $K \in \R$
and $N > 0$
such that $(\cM,\met_\mg, \mu_\mg) \in \RCD(K, N)$.
If $\cS \neq \cM$ is a closed set and 
$\mg \in \Ck{k}(\cM \setminus \cS)$, 
there exists a family of metrics
$\mg_t \in \Ck{k-1,1}$ on $\cM \setminus \cS$
evolving according to \eqref{Def:GM}
on $\cM \setminus \cS$. For two points $x, y \in \cM$
that are $\mg_t$-admissible,
the distance 
$\met_t(x,y)$ given by the $\RCD(K,N)$ Gigli-Mantegazza
 flow  is induced by $\mg_t$.
\end{theorem}

The \emph{divergence form} structure is  an essential  
feature that allows for the analysis of the defining
continuity equation \eqref{Def:E}.
In the compact case, 
it turns out that near every rough metric $\mg$, 
there is a smooth metric $\mgt$ in
a suitable $\Lp{\infty}$-sense.
Coupling this with the divergence
structure, we are able to \emph{perturb}
this problem to the study of a
divergence form operator with  bounded, measurable
coefficients on $\mgt$ of the form 
$$ -\divv_{\mgt} \hk^\mg_t(x,\mdot) \B \theta  \conn \phi_{t,x,v} = \theta \extd_x(\hk^\mg_t(x,\mdot))(v),$$ 
where $\B$ is a bounded, measurable $(1,1)$-tensor
transforming $\mgt$ to $\mg$ and
$\theta$ the Raydon-Nikodym derivative
of the two induced measures $\mu_\mgt$ and $\mu_\mg$.

This trick of ``hiding''
singularities in an operator has 
its origins to the investigations of
boundary value problems 
with low regularity boundary. 
For us, this philosophy
has a geometric reincarnation  arising from 
investigations of the Kato square
root problem on manifolds, with its origins
in a seminal paper  \cite{AKMc} by Axelsson, Keith and
McIntosh and more recently by 
Morris in \cite{Morris}, and Bandara and McIntosh in \cite{BMc}. 

In \S\ref{Sec:ER}, we study the existence 
and regularity of solutions to such continuity 
equations by using
spectral methods and PDE tools. We further emphasise that such 
equations allow for a certain \emph{disintegration} - 
that is, at each point $x$, we solve
a differential equation in the $y$ variable.
Eventually, we are concerned with objects
involving an integration in $y$, and hence, we are
able to allow weak solutions in $y$ while
being able to prove stronger regularity results
in $x$. 

While the reduction of a nonlinear problem 
to a pointwise linear one is a tremendous
boon to the analysis that we conduct
in this paper, there is a price to pay. 
The equation \eqref{Def:E} is nonlinear in $x$, and this
nonlinear behaviour requires
the analysis in $x$ of the family
of operators
$$x \mapsto \divv_{\mgt} \hk^\mg_t(x,\mdot)\B\theta \conn.$$
This is not as significant a disadvantage as one initially anticipates
as this opens up the possibility
to attacking regularity questions
by the means of operator theory.

One of the main points of this paper
is to illustrate how the regularity
properties of the flow \eqref{Def:GM}
relate to the regularity properties of the heat kernel. 
Theorem \ref{Thm:Main}
allows for the possibility of the evolving metric
to become \emph{less} regular than the original 
metric on the non-singular subset.
An inspection of the continuity equation
\eqref{Def:E} shows that the solution at a point $x$ 
depends on sets of full measure
potentially far away from this point.
Thus, it is possible that
singularities may resolve from 
smoothing properties emerging from the heat kernel. 
However, it is also possible
that potentially unruly behaviour
somewhere in heat kernel forces the flow to 
introduce additional singularities.
That being said, we show that 
for $k \geq 1$, 
$\Ck{k}$ metrics will continue to be 
$\Ck{k}$ under the flow. We discuss these results and surrounding
issues in greater depth in \S\ref{Sec:Res}.
\section*{Acknowledgements}

This research was conducted during the
``Junior Trimester Program on Optimal Transport'' at
the Hausdorff Research Institute for Mathematics
in Bonn, Germany. The authors acknowledge the gracious
support of this institution as well as the program.

The authors would like to thank Nicola Gigli for encouraging
conversation as well as 
Christian Ketterer, Martin Kell and Alex Amenta for helpful 
conversations and suggestions.
\section{Geometric singularities}

Throughout this paper, by 
the term \emph{geometric singularity},
we shall mean singularities that
arise in the metric $\mg$ of
a \emph{smooth manifold} $\cM$.
We allow such singularities 
to be a lack of differentiability
or even the lack of continuity.
To emphasise this point, we contrast this to 
pure topological singularities, 
which are singularities that
live in the topology and cannot be 
smoothed and transferred into 
the metric. 

Let $\cM$ to be a smooth manifold
(possibly non-compact) of dimension $\dim \cM = n$.
By this, we mean a second countable, Hausdorff
space that is locally Euclidean, with the
transition maps being smooth.

On an open subset $\Omega \subset \cM$, 
we write $\Ck{k,\alpha}(\Omega)$
($k \geq 0$ and 
$\alpha \in [0,1]$) to mean
$k$-times continuously differentiable functions bounded
\emph{locally} in coordinate patches inside $\Omega$, and where
the $k$-th partial derivatives are $\alpha$-H\"older
continuous \emph{locally}. 
We write $\Ck{k}(\Omega)$ instead
of $\Ck{k,0}(\Omega)$.  

Let $\Tensors[p,q]\cM$ denote the tensors
of covariant rank $p$ and contravariant rank $q$.
We write $\cotanb\cM = \Tensors[1,0]\cM$
and $\tanb\cM = \Tensors[0,1]\cM$,
the \emph{cotangent} and \emph{tangent} bundles
respectively. The bundle of differentiable 
$k$-forms are given by $\Forms[k]\cM$
and the exterior algebra
is given by $\Forms\cM = \oplus_{k=0}^n \Forms[k]\cM$,
where $\cM\times \R = \Forms[0]\cM$ 
(the bundle of functions), and $\cotanb\cM = \Forms[1]\cM$.

The differentiable
structure of the smooth manifold 
affords us with a differential operator
$\extd: \Ck{\infty}(\Forms[k]\cM) \to \Ck{\infty}(\Forms[k+1]\cM)$.
Indeed, this differential operator is dependent
on the differentiable structure we associate
to the manifold. We remark on this fact 
since, in dimensions higher than $4$, there
are homeomorphic differentiable structures
that are not diffeomorphic (cf. \cite{Milnor}
by Milnor and \cite{Freedman} by Freedman).
From this point   onward , we fix a differentiable
structure on $\cM$.
We shall only exercise interest in the 
case of $k = 0$ 
where $\extd: \Ck{\infty}(\cM) \to \Ck{\infty}(\cotanb\cM)$
and sometimes use the notation $\conn$ to denote
$\extd$.

We emphasise that a smooth manifold
also affords us with a measure structure.
We say that a set $A \subset \cM$
is measurable if for any 
chart $(\psi, U)$ with 
$U \intersect A \neq \emptyset$, we obtain 
that $\psi(U \intersect A) \subset \R^n$
is Lebesgue measurable. By second countability, 
this quantification can be made countable.
By writing $\Sect(\Tensors[p,q]\cM)$
we denote the sections of the
vector bundle $\Tensors[p,q]\cM$ with 
\emph{measurable} coefficients.

\subsection{Rough metrics}
\label{Sec:Rough}
In connection with investigating the
geometric invariances of the Kato square root problem,
Bandara introduced a notion of a \emph{rough
metric} in \cite{BRough}. This notion is of
fundamental importance to the rest of this paper
and therefore, in this sub-section,
we will describe some of the important features of such metrics.
We do not assume that $\cM$ is compact until later
in this section.
Let us first recall the definition of a rough metric.

\begin{definition} 
We say that a real-symmetric $\mg \in \Sect(\Tensors[2,0]\cM)$
is a \emph{rough metric} if 
for each $x \in \cM$,
there exists some chart  $(\psi,U)$ containing $x$
and a constant $C \geq 1$ (dependent on $U$),
such that, for $y$-a.e. in $U$,
$$C^{-1} \modulus{u}_{\pullb{\psi}\delta(y)} \leq
	\modulus{u}_{\mg(y)} \leq C \modulus{u}_{\pullb{\psi}\delta(y)},$$
 where $u \in \tanb_y\cM$, $\modulus{u}_{\mg(y)}^2 = \mg(u,u)$
and $\pullb{\psi}\delta$ is the pullback
of the Euclidean metric inside $\psi(U) \subset \R^n$. 
Such a chart is said to satisfy the 
\emph{local comparability condition}.
\end{definition}

It is easy to see that by taking $U$ to be 
the pullback of a Euclidean 
ball contained in a chart near $x$, every 
$\Ck{k,\alpha}$  metric (for $k \geq 0$ and 
$\alpha \in [0,1]$)  is a rough metric.

 Two rough metrics 
$\mg$ and $\mgt$ are said to be $C$-close (for $C \geq 1$)
if  
$$ C^{-1} \modulus{u}_{\mg(x)} \leq \modulus{u}_{\mgt(x)} \leq C \modulus{u}_{\mg(x)}$$
for almost-every $x \in \cM$.
If we assume that $\cM$ is compact, 
then it is easy to see that for any 
rough metric $\mg$, there exists a constant
$C \geq 1$ and a  smooth metric
$\mgt$ such that $\mg$ and $\mgt$ are $C$-close.
Two continuous metrics are $C$-close
if the $C$-close  condition 
above holds everywhere. 
Moreover, we note the following. 
Its proof is given in Proposition 10 in \cite{BRough}.
\begin{proposition}
\label{Prop:OpExist}
Let $\mg$ and $\mgt$ be two rough metrics that are 
$C$-close. Then, there exists $\B \in \Sect(\cotanb\cM \tensor \tanb\cM)$
such that it is real, symmetric,  almost-everywhere positive, invertible, and
$$\mgt_x(\B(x)u,v) = \mg_x(u,v)$$
for almost-every $x  \in \cM$. Furthermore, 
for almost-every $x \in \cM$, 
$$C^{-2} \modulus{u}_{\mgt(x)} \leq \modulus{\B(x)u}_{\mgt(x)} \leq C^2 \modulus{u}_{\mgt(x)},$$
and the same inequality  holds  with $\mgt$ and $\mg$ interchanged.
If $\mgt \in \Ck{k}$ and $\mg \in \Ck{l}$ (with $k, l \geq 0$),
then the properties of $\B$ are valid for all $x \in \cM$ and
$\B \in \Ck{\min\set{k,l}}(\cotanb\cM \tensor \tanb\cM).$
\end{proposition}

A rough metric always induces a measure described
by the expression 
$$ d\mu_\mg(x) = \sqrt{\det (\mg_{ij}(x))}\ d\Leb(x)$$
inside a locally comparable chart. The well-definedness
of this expression is verified just as 
in the case of a $\Ck{k,\alpha}$ metric.
This measure can easily be proven to be Borel
and finite on compact sets. 
The notion of measurable which we have
defined agrees with the notion of $\mu_\mg$-measurable
obtained via a rough metric.
Moreover, the following holds for two $C$-close
metrics.
\begin{proposition}
\label{Prop:MeasRep}
Let $\mg$ and $\mgt$ be $C$-close for some $C \geq 1$.
Then, the measure
$d\mu_\mg(x) = \sqrt{\det \B(x)}\ d\mu_\mgt(x)$
for $x$-a.e., and
$C^{-\frac{n}{2}} \mu_\mg \leq \mu_\mgt \leq C^{\frac{n}{2}} \mu_\mg.$
 Moreover, if $\mgt$ is continuous, then the measure $\mu_\mg$ is Radon. 
\end{proposition}
\begin{proof}
The first part of the statement is proved as  Proposition 11 in \cite{BRough}.
We prove that $\mu_\mg$ is Radon by using the fact
that $\mu_\mgt$ is inner regular.
That is, for a Borel
$B \subset \cM$ and every $\epsilon > 0$
there exists $K_\epsilon \subset B$ such that
$\mu_\mgt(B) - \mu_\mgt(K_\epsilon) \leq \epsilon.$
Therefore,
\begin{multline*} 
\mu_\mg(B) - \mu_\mg(K_\epsilon)
	= \int (\ind{B} - \ind{K_\epsilon})\ d\mu_\mg
	\leq C^{n/2} \int (\ind{B} - \ind{K_\epsilon})\ d\mu_\mgt \\
	\leq C^{n/2} (\mu_\mgt(B) -  \mu_\mgt(K_\epsilon))
	\leq C^{n/2} \epsilon, 
\end{multline*}
where the first inequality follows from the fact
that $\ind{B} - \ind{K_\epsilon} \geq 0$. 
Thus, $\mu_\mg(B) = \sup_{K \Subset B} \mu_\mg(K).$
\end{proof}

We remark that throughout the paper, when we say that
$(\cM,\mg)$ \emph{induces a length structure}, 
we mean that between any two points $x, y \in \cM$
there exists an absolutely continuous curve  $\gamma:I \to \cM$ with 
$\gamma(0)= x$, $\gamma(1) = y$ such that 
$$ 0 < \int_{I} \modulus{\dot{\gamma}(t)}^2_{\mg(\gamma(t))} < \infty.$$
Then, the induced distance $\met_\mg(x,y)$ is simply given
as in the smooth case by taking an infimum
over all curves between such points of the square root of this quantity.  

\subsection{$\Lp{\infty}$-metrics and metrics of divergence form operators}

The goal of this subsection is to 
illustrate the connections
of rough metrics to other low-regularity metrics
that are often mentioned in folklore.
In fact, we shall see that as a virtue of
compactness, these notions are indeed equivalent.
This section is intended as motivation for us
considering rough metrics and can be safely omitted.

First, we highlight the following simple lemma. 
\begin{lemma}
\label{Lem:RChar}
Suppose that $\mg \in \Sect(\Tensors[2,0]\cM)$
is symmetric and that there exists a
smooth metric $\mh$ and $C \geq 1$ such that
$$ C^{-1} \modulus{u}_{\mh(x)} \leq 
	\modulus{u}_{\mg(x)} \leq C \modulus{u}_{\mh(x)}$$
for almost-every $x \in \cM$. Then $\mg$ is a rough metric.
\end{lemma}
\begin{proof}
Fix $x \in \cM$ and let $(\psi, U)$ be
a chart near $x \in \cM$. Let $V = \psi^{-1}(B_{r}(\psi(x)))$, 
where $B_r(\psi(x)) \subset \R^n$ is a Euclidean 
ball with $r > 0$ chosen such that $\close{B_r(\psi(x))} \subset \psi(U)$.
Then, by virtue of the smoothness of $\mh$ and since $\close{V}$
is compact, we obtain some $C_{V} \geq 1$ such that
$$ C_{V}^{-1} \modulus{u}_{\pullb{\psi}\delta(y)} 
	\leq \modulus{u}_{\mh(y)} \leq C_V \modulus{u}_{\pullb{\psi}\delta(y)},$$
for all $y \in V$.
On combining this with our hypothesis, we find that
for almost-every $y \in V$,
 $$ (C_{V}C)^{-1} \modulus{u}_{\mh(y)} \leq 
	\modulus{u}_{\mg(y)} \leq C_VC \modulus{u}_{\mh(y)}.$$
That is, $\mg$ is a rough metric.
\end{proof}

Next, we define the following notion of
an $\Lp{\infty}$-metric.
\begin{definition}
We say that a real-symmetric $\mg \in \Sect(\Tensors[2,0]\cM)$ 
is an $\Lp{\infty}$-metric on $\cM$ if:
\begin{enumerate}[(i)]
\item there exists a $\mg^{-1} \in \Sect(\Tensors[0,2]\cM)$
	inverse to $\mg$, by which we mean that
	writing $G = (\mg^{ij}(x))$ and $G' = (\mg^{-1}_{ij}(x))$,
	$GG' = G'G = \iden$ for almost-every $x \in \cM$, and
\item there exists a smooth metric $\mh$ and constants
	$\Lambda_1, \Lambda_2 > 0$ such that
	$\modulus{\mg}_{\mh} \leq \Lambda_1$ 
	and $\modulus{\mg^{-1}}_{\mh} \leq \Lambda_2$.
\end{enumerate} 
\end{definition}

We prove that an $\Lp{\infty}$-metric is indeed
a rough metric.
 
\begin{proposition}
\label{Prop:LpinfR}
An $\Lp{\infty}$-metric $\mg$ is a rough metric.
It is $(\max\set{\Lambda_1 n, \Lambda_2 n})$-close to a 
smooth metric $\mh$.
\end{proposition}
\begin{proof}
Fix $x \in \cM$ in which the inequalities in 
the definition of an $\Lp{\infty}$-metric is valid.
Let $\set{e^i}$ be a frame 
for $\tanb_x \cM$ so that  $\mh_{ij}(x) = \delta_{ij}$.
Let $G = (\mg^{ij}(x))$
Then, note that
$$ \Lambda_1 \geq \modulus{\mg(x)}_{\mh(x)}^2 
	= \mg^{ij}(x)\mg^{mn}(x) \mh_{im}(x) \mh_{jn}
	= \sum_{ij} \modulus{\mg^{ij}}^2.$$
Now, let $u \in \tanb_x \cM$, and then 
\begin{multline*}
\modulus{u}_{\mg(x)}^2 
	= \mg(x) (u,u)
	= \mg^{ij} u_i u_j
	\leq \sum_{j}  \modulus{ \sum_{i} \mg^{ij} u_i u_j} \\
	\leq \sum_{j} \cbrac{\sum_{i} \modulus{\mg^{ij} u_j}^2}^{\frac{1}{2}}
			\cbrac{\sum_{i} \modulus{u_i}^2}^{\frac{1}{2}},
\end{multline*}
where the last inequality follows form the Cauchy Schwarz inequality.
Now, by our previous calculation, we have that
$\modulus{\mg^{ij}}^2 \leq \Lambda_1$, and hence, 
$$ 
\sum_j \sum_{i} \modulus{\mg^{ij} u_j}^2 \leq \Lambda_1 n \sum_j \modulus{u_j}^2 
	= \Lambda_1 n \modulus{u}_{\mh(x)}^2.$$
That is,
$\modulus{u}_{\mg(x)}  \leq \sqrt{ \Lambda_1 n }  \modulus{u}_{\mh(x)}.$
Applying this with $\mg^{-1}$ in place of $\mg$ and
$\mh^{-1}$ in place of $\mh$, we further obtain that
$\modulus{u}_{\mg^{-1}(x)}  \leq \sqrt{ \Lambda_1 n }  \modulus{u}_{\mh^{-1}(x)}.$

Now, we note that $\tilde{G} = (\mg^{-1}_{ij}(x)) = G^{-1}$.
Since $G$ is symmetric, let $G = PD\tp{P}$, its eigenvalue
decomposition. Indeed, $D = \diag(\lambda_i)$ and $\lambda_i > 0$
since $G$ is invertible. On letting $D^{-1} = \diag(\sigma_i)$,
note that $\sigma_n = \lambda_1$ and 
we have that $\sigma_n \leq \Lambda_2 n$,
which means that $\lambda_1 \leq \frac{1}{\Lambda_2 n}$.
So, now 
\begin{multline*}
 \modulus{u}_{\mg(x)}^2
	= \tp{u}Gu = \tp{u}PD\tp{P}u = \tp{(\tp{P}u)}D(\tp{P}u)
	= \tp{\tilde{u}}D \tilde{u} \\
	=  \sum_{i} \lambda_i \modulus{\tilde{u}_i}
	\geq \min\set{\lambda_i} \sum_i \modulus{\tilde{u}_i} 
	\geq \frac{1}{\Lambda_2 n} \modulus{\tp{P}u}
	= \frac{1}{\Lambda_2 n}  \modulus{u}^2,
\end{multline*}
since $P$ is an orthonormal matrix.
That is, we have shown that for almost-every $x \in \cM$, 
and every $u \in \tanb_x \cM$, 
$$ (\Lambda_2n)^{-1} \modulus{u}_{\mg(x)} \leq \modulus{u}_{\mh(x)} 
	\leq (\Lambda_1 n)^{-1} \modulus{u}_{\mh(x)}.$$
Thus, by Lemma \ref{Lem:RChar}, $\mg$ is a rough metric.
\end{proof}

Another class of metrics we consider
are metrics arising from coefficients of 
elliptic operators in divergence form.
In particular, see the paper \cite{SC}
by Saloff-Coste, where the author 
explicitly considers this class of metrics,
although he makes a qualitative assumption that
the coefficients are smooth.

Fix some smooth metric $\mh$
and let $A \in \Sect(\Tensors[1,1]\cM,\mh)$
be real-symmetric.
Consider the real-symmetric form 
$J_A[u,v] = \inprod{A\conn u, \conn v}_{\mh}$. 
In order for this to define an elliptic
operator, the natural assumption
is to ask that there exist
$\kappa_1, \kappa_2 > 0$ such that 
$$ \mh(Au,u)(x) \geq \kappa_1 \modulus{u}_{\mh(x)}^2
\quad\text{and}\quad\modulus{A}_{\mh(x)} \leq \kappa_2,$$
for almost-every $x \in \cM$.
For the sake of nomenclature, let us say that
the coefficients $A$ are \emph{elliptic}
if this condition is met.
Under these conditions,
the self-adjoint operator
associated to $J_A$ is
$\Div_A u = \divv_{\mh} A \conn u$.
Such operators have been amply studied
in the literature.

Let us now define a metric associated
to elliptic coefficients $A$ by writing
$\mg(u,v) = \mh(Au, v)$. Then, we
have the following proposition. 

\begin{proposition}
A metric $\mg$ induced from elliptic coefficients $A$
via a smooth metric $\mh$ is a rough metric. 
The metric 
$\mg$ is $(\max\set{\kappa_1, \kappa_2})$-close to $\mh$.
\end{proposition}
\begin{proof}
By virtue of the fact that $A$
are elliptic coefficients, we immediately obtain
that for almost-every $x\in \cM$, 
$\modulus{u}_{\mg(x)}^2 = \mh(A u, u) \geq \kappa_1^2 \modulus{u}_{\mh(x)}^2$
for every $u \in \tanb_x \cM$.

For the upper bound, fix an $x$ where
the ellipticity inequality is valid, and
choose a frame so that
let $\mh_{ij}(x) = \delta_{ij}$. 
Then, we have that 
$\modulus{A}_{\mh(x)}^2 = \sum_{ij} \modulus{A^i_j}^2.$
Then, 
$$ \modulus{u}_{\mg(x)}^2 = \mh_x(Au,u)
	\modulus{Au}_{\mh(x)} \modulus{u}_{\mh(x)}.$$
Now, 
$$
\modulus{Au}_{\mh(x)}^2
	= \sum_{j} \modulus{\sum_i A^i_j u_i}^2
	\leq \sum_{j} \cbrac{\sum_i \modulus{A^i_j}^2}
		\cbrac{\sum_i \modulus{u_i}^2}
	\leq \kappa_2^4 \modulus{u}_{\mh(x)}^2.$$
Thus, 
$\modulus{u}_{\mg(x)}^2 \leq \kappa_2^2 \modulus{u}_{\mh(x)}^2.$
By invoking Lemma \ref{Lem:RChar}, we obtain that
$\mg$ is a rough metric.
\end{proof}

As a finale, on collating our
results here, we present the following proposition. 

\begin{proposition}
We have the following:
\begin{enumerate}[(i)] 
\item every rough metric that is close to a smooth one
is an $\Lp{\infty}$-metric,

\item every $\Lp{\infty}$-metric is a metric
induced via elliptic coefficients,

\item every metric induced via elliptic coefficients
	is an $\Lp{\infty}$-metric.
\end{enumerate}
If $\cM$ is compact, then all these notions are
equivalent.
\end{proposition}
\begin{proof}
For (i), suppose that
$\mg$ is a rough metric and that it is $C$-close
to a smooth metric $\mh$. Then,
by Proposition \ref{Prop:OpExist}, we obtain 
a $B \in \Sect(\Tensors[1,1]\cM)$ so that
$\mg(u,v) = \mh(Bu,v)$ and 
$$C^{-2}\modulus{u}_{\mh(x)} \leq  \modulus{B(x)u}_{\mh(x)} 
	\leq C^2 \modulus{u}_{\mh(x)},$$
for almost-every $x$. Fix an $x$ where
this inequality is valid and choose a
$\mh$ orthonormal frame $\set{e_i}$
at $x$. Then, 
$$
\modulus{\mg(x)}_{\mh(x)}^2 
	= \sum_{ij} \modulus{\mg^{ij}(x)}^2
	= \sum_{ij} \modulus{\mh(Be_i, e_j)}
	\leq \sum_{ij} \modulus{B e_i}_{\mh(x)} \modulus{e_j}_{\mh(x)}
	\leq n^2 C^2.$$
So, $\modulus{\mg}_{\mh} \leq n^2 C^2$
almost-everywhere. 
Since $\mg^{-1}(u,v) = \mh^{-1}(B^{-1} u,v)$, 
a similar calculation shows that
$\modulus{\mg^{-1}}_{\mh^{-1}} \leq n^2 C^2.$

To prove (ii), suppose that $\mg$ is an $\Lp{\infty}$-metric.
Then, we have shown in Proposition \ref{Prop:LpinfR}
that it is a rough metric that is close to a 
smooth metric $\mh$. Hence, by Proposition \ref{Prop:OpExist}, 
we have $B \in \Sect(\Tensors[1,1]\cM)$
which can easily be checked to satisfy ellipticity.
Hence, $\mg(u,v) = \mh(Bu,v)$, i.e., it is 
a metric induced by elliptic coefficients.
Then, it is a rough metric that is $C$-close to a smooth
one and by (i), we obtain that it is an $\Lp{\infty}$-metric.

If further we assume that $\cM$ is a compact manifold, 
then near every rough metric $\mg$, there is a
smooth metric $\mh$, and hence, by (i), we
obtain that every rough metric is $\Lp{\infty}$
or equivalently, defined via elliptic coefficients.
\end{proof}

In particular, this proposition gives credence
to the notion of a rough metric since
 it is a sufficiently general notion that is
able to capture the behaviour of these other 
aforementioned low regularity metrics.  

\subsection{Lebesgue and Sobolev space theory}

A more pertinent feature of rough metrics is
that they admit a Sobolev space theory.
In order to make our exposition shorter
and more accessible, from here on, we assume
that $\cM$ is \emph{compact}.
First, we note that since $(\cM,\mu_\mg)$
is a measure space, we obtain a Lebesgue
theory.
Let $\Lp{p}(\Tensors[p,q]\cM,\mg)$ denote the
$p$-integrable Lebesgue spaces over the bundle of $(p,q)$ tensors.
We write $\Lp{p}(\cM,\mg)$ for the case that $p = q = 0$.
 We  quote the following result which is listed 
as Proposition 8 in \cite{BRough}.

\begin{proposition}
\label{Prop:OpSobRough}
For a rough metric $\mg$,
$\conn[p]: \Ck{\infty}\intersect \Lp{p}(\cM) \to \Ck{\infty}\intersect \Lp{p}(\cotanb\cM)$
and $\conn[c]: \Ck[c]{\infty}(\cM) \to \Ck[c]{\infty}(\cotanb\cM)$
given by $\conn[p] = \extd$ and $\conn[c]  = \extd$ on the respective domains
are closable, densely-defined operators. 
\end{proposition}

As a consequence,  we define the  Sobolev spaces as function spaces
by writing 
$\Sob{1,p}(\cM) = \dom(\close{\conn[p]})$ and 
$\Sob[0]{1,p} = \dom(\close{\conn[c]})$.
\begin{proposition}
\label{Prop:RoughP}
Let $\mg$ and $\mgt$ be two $C$-close rough metrics
on a compact manifold $\cM$.
Then,
\begin{enumerate}[(i)]
\item whenever $p \in [1, \infty)$, 
	$\Lp{p}(\Tensors[r,s]\cM,\mg) = \Lp{p}(\Tensors[r,s]\cM, \mgt)$ with
	$$C^{-\cbrac{r+s + \frac{n}{2p}}} \norm{u}_{p,\mgt}   \leq \norm{u}_{p,\mg} \leq C^{r+s + \frac{n}{2p}} \norm{u}_{p, \mgt},$$
\item for $p = \infty$,
	$\Lp{\infty}(\Tensors[r,s]\cM,\mg) = \Lp{\infty}(\Tensors[r,s]\cM, \mgt)$
	with
	$$C^{-(r+s)} \norm{u}_{\infty, \mgt} \leq \norm{u}_{\infty,\mg} \leq C^{r+s} \norm{u}_{\infty,\mgt},$$
\item the Sobolev spaces 
	$\Sob{1,p}(\cM,\mg) = \Sob{1,p}(\cM, \mgt) = \Sob[0]{1,p}(\cM, \mg) = \Sob[0]{1,p}(\cM, \mgt)$
	with 
	$$C^{-\cbrac{1 + \frac{n}{2p}}} \norm{u}_{\Sob{1,p},\mgt} 
		\leq \norm{u}_{\Sob{1,p},\mg} \leq C^{{1 + \frac{n}{2p}}} \norm{u}_{\Sob{1,p},\mgt},$$
\item the Sobolev spaces
	$\Sob{\extd,p}(\cM,\mg) = \Sob{\extd,p}(\cM,\mgt)$ with 
	$$C^{-\cbrac{n + \frac{n}{2p}}} \norm{u}_{\Sob{\extd, p}, \mgt}
		\leq \norm{u}_{\Sob{\extd,p},\mg}
		\leq C^{{n + \frac{n}{2p}}} \norm{u}_{\Sob{\extd, p}, \mgt},$$
\item the divergence operators satisfy 
	$\divv_{\mg} = \uptheta^{-1}\divv_{\mgt}\uptheta \B$.
\item the Laplacians satisfy
	$\Lap_{\mg} = -\uptheta^{-1} \divv_{\mgt}\uptheta \B \conn$.
\end{enumerate} 
\end{proposition}

We emphasise (iii), which demonstrates
that $\Sob{1,2}(\cM,\mg) = \Sob{1,2}(\cM,\mgt)$
for any rough metric $\mg$, since, as
we have aforementioned, compactness guarantees
the existence of a smooth metric $\mgt$
 that is  $C$-close to $\mg$.
\section{Main results and applications}
\label{Sec:Res}

\subsection{Existence and regularity of the flow}

The broader perspective underpinning our analysis in
this paper is to relate the regularity of the
heat kernel to the regularity of the Gigli-Mantegazza flow. Indeed, this
is to be expected simply from inspection 
of the main governing  equation 
\eqref{Def:E} for this flow. 

In \S\ref{Sec:ER}, we  consider
$\Lp{\infty}$-coefficient differential operators
on smooth manifolds, and we obtain solutions
to more general equations similar to \eqref{Def:E}.
Furthermore, we conduct
operator theory on operators
of the type $x \mapsto \divv_{\mg}\omega_x \conn$
in order to define a notion of derivative 
that is weak enough to account for the lack of
regularity of the coefficients $\omega_x$  but 
sufficiently strong enough to be useful 
to demonstrate the regularity of the
the flow \eqref{Def:GM}.
In \S\ref{Sec:Flow}, we prove some 
auxiliary facts needed to ensure that
\eqref{Def:GM} indeed does define a
Riemannian metric, and on coupling our
main results from \S\ref{Sec:ER},
we obtain the following theorem.
It is the most general geometric result
that we showcase in this paper. 
Its proof can be found in \S\ref{Sec:GMFlow}.
 
\begin{theorem}
\label{Thm:GMFlow}
Let $\cM$ be a smooth, compact manifold
and $\mg$ a rough metric. Suppose that
the heat kernel $(x,y) \mapsto \hk^\mg_t(x,y) \in \Ck{0,1}(\cM^2)$
and that on an open set $\emptyset \neq \cN$, 
$(x,y) \mapsto\hk^{\mg}_t(x,y) \in \Ck{k}(\cN^2)$
where $k \geq 2$.
Then, for $t > 0$, $\mg_t$ is a Riemannian metric
on $\cN$ of regularity $\Ck{k - 2,1}$.
\end{theorem}

We remark that allowing for a 
Lipschitz heat kernel is
neither a restriction nor is it too general.
We will see in the following section that
the most important class of objects
we consider, namely when $(\cM,\mg)$
is an $\RCD(K,N)$ space, will admit such a heat kernel.

The reader may find it curious
that, even though we assume
that the heat kernel is $\Ck{k}$
away from the 
singular region,
and only a single derivative of
the heat kernel appears in the source
term of  \eqref{Def:E}, we are only able
to assert that the  resulting  flow is $\Ck{k-2,1}$.
In a sense, it is because the global regularity 
of the heat kernel, which is only Lipschitz, 
becomes significant in proving the 
continuity of the $(k-1)$-th  partial derivatives.
We remark that it may be possible 
to assert this continuity 
performing the operator theory of
$x \mapsto \divv_{\mg} \hk^\mg_t(x,\mdot)\conn$ 
with greater care than we have done.  

For a $\Ck{1}$ global heat kernel,
we are able to assert that the
$(k-1)$-th partial derivatives are indeed
continuous. This is the content of
the following theorem.
Note that, unlike Theorem \ref{Thm:GMFlow}
where we assumed that the 
heat kernel was at least twice
continuously differentiable on 
the non-singular part,
we allow for heat kernels with 
only a single derivative 
on the non-singular region.
This theorem is an immediate consequence
of Theorem \ref{Thm:BetReg} in \S\ref{Sec:HighReg}.

\begin{theorem}
\label{Thm:GMFlowReg}
Let $\cM$ be a smooth, compact manifold
and $\mg$ a rough metric. Suppose that
the heat kernel $(x,y) \mapsto \hk^\mg_t(x,y) \in \Ck{1}(\cM^2)$
and that on an open set $\emptyset \neq \cN$, 
$(x,y) \mapsto\hk^{\mg}_t(x,y) \in \Ck{k}(\cN^2)$
where $k \geq 1$.
Then, for $t > 0$, $\mg_t$ is a Riemannian metric
on $\cN$ of regularity $\Ck{k - 1}$.
\end{theorem}

We remark that for the case $k = 1$, Bandara in \cite{BCont}
asserts the continuity of $x \to \mg_t(x)$ on $\cN$
without imposing any global regularity assumptions on 
$(x,y) \mapsto \hk_t^\mg(x,y)$. This is accomplished by reducing 
this assertion to a \emph{homogeneous} Kato square root problem
that he proves in this context. 

Typically, in an open region where
the metric is $\Ck{k}$ for $k \geq 1$,
we expect the heat kernel to improve
to $\Ck{k+1}$. This is an immediate
consequence of the fact that the region is
open, and  because  we can write the Laplacian
via a change of coordinates as a
non-divergence form equation with $\Ck{k-1}$
coefficients. We then obtain regularity
via Schauder theory. 
This analysis is conducted in \S\ref{Sec:HKRegMet}.
 By considering  the situation where
a rough metric improves to a $\Ck{k}$
metric away from a closed singular
set, Theorem \ref{Thm:GMFlow} yields a
metric tensor away from the singular
set that is of regularity 
$\Ck{k-1,1}$. That is, the
resulting flow may be \emph{more} 
singular than the initial metric 
inside such a region. However,
by coupling the results of
\S\ref{Sec:HKRegMet} with 
Theorem \ref{Thm:GMFlowReg},
we are able to assert that the
flow remains at least
as regular as the initial metric.

\begin{theorem}
Let $\cM$ be a smooth, compact manifold
and $\mg \in \Ck{k}$ for $k \geq 1$.
Then, the flow $\mg_t \in \Ck{k}$ for each 
$t > 0$. 
\end{theorem}

 We emphasise that we are not providing
sharp regularity information via these 
theorems.  That is,
we are unable to assert that if the initial 
metric is $\Ck{k}\setminus \Ck{k-1}$,
then the resulting flow is also
$\Ck{k}\setminus \Ck{k-1}$. In fact, 
this may not be the case, it may be
possible that in some instances, 
the flow $\mg_t$ improves in regularity.
These are interesting open questions beyond
the scope of this paper. 

\subsection{Applications to geometrically singular spaces}
\label{Sec:Sing}

In this subsection, we consider geometric
applications of Theorem \ref{Thm:GMFlow},
particular to spaces with geometric singularities.

As a start, we describe our notion of a 
\emph{geometric conical singularity}. 
For that, let us first describe
the $n$-cone of radius $r$ and height
$h$ by 
$$ \cC(r,h) = \set{ (x,t) \in \R^{n+1}: \modulus{x} 
	= \frac{r}{h}(h - t): t \in [0,h]}.$$
With this notation in hand, we 
define the following. 

\begin{definition}[Geometric conical singularities]
Let $\cM$ be a smooth manifold and $\mg$ a rough metric.
Let $\set{p_1, \dots, p_k} \subset \cM$
and suppose there exists a charts
$(\psi_i, U_i)$ mutually disjoint such that
$\mg \in \Ck{k}(\cM \setminus \union_{i} \close{U}_i)$.
Moreover, suppose for each $i$, there
is a Lipeomorphism
$F_i:U_i \to \cC(r_i, h_i) \subset \R^{n+1}$ 
which improves to a $\Ck{k+1}$ diffeomorphism (for $k \geq 1$)
on $U_i\setminus \set{p_i}$
and that $\mg = \pullb{F}_i\inprod{\mdot,\mdot}_{\R^{n+1}}$
inside $U_i$.
Then, we say that $(\cM,\mg)$ is $\Ck{k}$-geometry
with \emph{geometric
conical singularities} at points $\set{p_1, \dots, \p_k}$.
\end{definition} 

A direct consequence of Theorem \ref{Thm:Main}
is then the following.

\begin{corollary}
Let $(\cM,\mg)$ be a $\Ck{k}$-geometry for $k \geq 1$ with 
geometric conical singularities at $\set{p_1, \dots, p_k}$.
Then, the flow $\mg_t \in \Ck{k-1, 1}(\cM \setminus \set{p_1, \dots, p_k})$ and 
the induced metric coincides everywhere with the  flow of the metrics $d_t$ 
for $\RCD(K,N)$ spaces defined by Gigli-Mantegazza.
\end{corollary}

Moreover, we are able to flow the sphere with a
conical pole. See \S\ref{Sec:Examples} for the construction and proof.

\begin{corollary}[Witch's hat sphere]
\label{Cor:Witch}
Let $(\Sph^n, \mg_{\witch})$ be the sphere with 
a cone attached at the north pole.
Then, the flow $\mg_t \in \Ck{\infty}$ at least on the region away
from the north pole and
the induced metric agrees with the flow of the metrics $d_t$ 
for $\RCD(K,N)$ spaces defined by Gigli-Mantegazza.
\end{corollary}

Also, as a consequence of Theorem \ref{Thm:Main}, we
are able to consider the $n$-dimensional box
in Euclidean space. Again, its proof is
contained in \S\ref{Sec:Examples}.

\begin{corollary}
\label{Cor:Box}
Let $(B,\mg)$ be an $n$-box.
Then, the flow $\mg_t \in \Ck{\infty}(\cM\setminus \cS)$,
where $\cS$ is the set of edges and
corners, induces the same distance
as $\met_t$, the $\RCD(K,N)$ Gigli-Mantegazza
flow, for $\mg_t$-admissible points.
\end{corollary}

We remark that a shortcoming of our analysis of
the flow for the box is that we have
not classified the $\mg_t$-admissible points.
We do not expect this analysis to be
straightforward since it involves understanding
whether the flat pieces are preserved in some
way under $\met_t$. However, since 
we have supplied a metric $\mg_t$ away from 
a set of measure zero, we expect it to be
possible to induce $\met_t$ via this
metric on a very large part of the box.
\section{Elliptic problems and regularity}
\label{Sec:ER}

Throughout this section, 
let us fix $\emptyset \neq \cN \subset \cM$ 
to be an  open  subset  of $\cM$.
To study the flow of Gigli-Mantegazza, 
we study a slightly more general 
elliptic problem
than \eqref{Def:E}.
Let $\omega \in \Ck{0,1}(\cM^2)$
be a function satisfying $\omega(x,y) > 0$
for all $x,y \in \cM$. 
Moreover,  let $x \mapsto \omega(x,\mdot) \in \Ck{k}(\cN)$
where $k \geq 1$.
For convenience, 
we write $\omega(x,\mdot) = \omega_x$.
Then, for $\eta \in \Lp{2}(\cM)$
we want to solve for $\phi \in \Sob{1,2}(\cM)$
satisfying the equation
\begin{equation} 
\tag{F}
\label{Def:F}
-\divv_\mg \omega_x \conn \phi = \eta.
\end{equation}

In this section, we establish existence, uniqueness
and regularity (in $x$)
for this equation.

\subsection{$\Lp{\infty}$-coefficient divergence form operators on smooth metrics}
\label{Sec:Linf}

Due to the lack of regularity  of an  arbitrary 
rough metric, we are forced to 
solve the the problem \eqref{Def:F} via
perturbation  to a smooth metric.  Indeed, 
this is not a small perturbation result as  
we are not guaranteed the existence of arbitrarily 
close smooth metrics to a rough metric.
Instead, we will have to contend
ourselves to studying elliptic PDE
of the form $-\divv A \conn$, where
$A$ are only bounded measurable coefficients
defining an elliptic problem.

More precisely, 
throughout this subsection, we fix $\cM$ to be a smooth compact manifold and $\mgt$
to be a smooth metric. 
Let $A \in \Lp{\infty}(\Tensors[1,1]\cM) = \Lp{\infty}(\bddlf(\Tensors[1,0]\cM))$
be real-symmetric. That is, for almost-every $x$, in coordinates, 
$A(x)$ can be written as a symmetric matrix with real coefficients.
Further, we assume that there exists $\kappa > 0$
such that $\inprod{Au,u} \geq \kappa \norm{u}^2$. 
That is, $A$ is bounded below. 
Moreover, we quantify 
the $\Lp{\infty}$ bound for
$A$ via assuming there exists $\Lambda > 0$ 
satisfying
$\inprod{Au,u} = \norm{\sqrt{A}}^2  \leq \Lambda$

Let $J_A:\Sob{1,2}(\cM) \times \Sob{1,2}(\cM) \to \R_+$ be given by 
$$ J_A[u,v] = \inprod{A\conn u, \conn v} 
	= \int_{\cM} \mgt_x(A(x)\conn u(x), \conn v(x))\ d\mu_\mgt(x).$$
Then, by the lower bound on $A$, we obtain that
$J_A[u,u] \geq \kappa{ \norm{\conn u}^2}$.
The Lax-Milgram theorem then yields a unique,  closed, densely-defined
self-adjoint operator $\Div_A$ with domain $\dom(\Div_A) \subset \Sob{1,2}(\cM)$
such that
$ J_A[u,v] =  \inprod{\Div_A u, v} $
for $u \in \dom(\Div_A)$ and $v \in \Sob{1,2}(\cM)$.
Moreover, since the form $J_A$ is real-symmetric
due to the fact that $A$ are real-symmetric coefficients, this theorem
further yields that $\dom(\sqrt{\Div_A}) = \Sob{1,2}(\cM)$
and $J_A[u,v] = \inprod{\sqrt{\Div_A}u, \sqrt{\Div_A}v}.$
The uniqueness then allows us to assert that
$\Div_A = -\divv_{\mgt} A \conn$.

For the remainder
of this section, we rely on some facts from 
the spectral theory of sectorial and, more particularly, 
self-adjoint operators.
We refer the reader to \cite{Kato} by Kato, \cite{ADM}
by Albrecht, Duong and McIntosh,  \cite{GT} 
by Gilbarg and Trudinger and 
\cite{CDMcY} by Cowling, Doust, McIntosh and Yagi for a more detailed
exposition on the connection between PDE and
spectral theory.

As a first, we establish a spectral splitting for $\Div_A$.
We recall that for $u \in \Lp[loc]{1}(\cM)$, 
 $u_{B} = \fint_{B} u\ d\mu_\mgt$ for $B \subset \cM$ a Borel set. 

\begin{proposition}
The space $\Lp{2}(\cM) = \nul(\Div_A) \oplus^{\perp} \close{\ran(\Div_A)}$
and the operator  $\Div_A$ restricted to either $\nul(\Div_A)$ or
$\close{\ran(\Div_A)}$ preserves each space respectively.
Moreover, $\nul(\Div_A) = \nul(\conn)$ and
$\close{\ran(\Div_A)} = \set{u \in \Sob{1,2}(\cM): \int_{M} u\ d\mu_{\mgt} = 0}.$
\end{proposition}
\begin{proof}
The fact that the operator splits the space orthogonally to 
$\nul(\Div_A)$ and $\close{\ran(\Div_A)}$, and that its restriction
to each of these spaces preserves the respective space
is a direct consequence of the fact that $\Div_A$ is self-adjoint.
 See Theorem 3.8 in \cite{CDMcY}.

First, let us apply this same argument to the operator
$\sqrt{\Div_A}$, so that
we obtain the splitting 
$\Lp{2}(\cM) = \nul(\sqrt{\Div_A}) \oplus^\perp \close{\ran(\sqrt{\Div_A})}$.
Now, fix $u \in \nul(\sqrt{\Div_A})$. Then, $\sqrt{\Div_A}u = 0$
which implies that $u \in \dom(\Div_A)$ and $\Div_A u = \sqrt{\Div_A} (\sqrt{\Div_A}u) = 0$.
Hence, $u \in \nul(\Div_A)$.
For the reverse inclusion, suppose that $0 \neq u \in \nul(\Div_A)$.
Then, $\sqrt{\Div_A} \sqrt{\Div_A} u = 0$. That is, 
$\sqrt{\Div_A}u \in \nul(\sqrt{\Div_A})$. 
But since $\sqrt{\Div_A}$ preserves $\nul(\sqrt{\Div_A})$,
$u \in \nul(\sqrt{\Div_A})$. Thus, $\nul(\Div_A)= \nul(\sqrt{\Div_A})$.
Now, since we have that for all $u \in \Sob{1,2}(\cM)$,
 $\kappa \norm{\conn u}^2 \leq \norm{\sqrt{\Div_A}u}^2 \leq \Lambda \norm{\conn u}^2$,
we obtain that $\nul(\conn) = \nul(\sqrt{\Div_A})$
and hence, $\nul(\conn) = \nul(\Div_A)$.

Now, recall that $(\cM,\mgt)$ admits a \Poincare
inequality: there exists $C > 0$ such that
$$ \norm{u - u_{\cM}} \leq C \norm{\conn u},$$
for every $u \in \Sob{1,2}(\cM)$.
Note then that, if $u \in \Sob{1,2}(\cM)$ 
and $\int_{\cM} u\ d\mu_{\mgt} = 0$, then 
$\norm{u} \leq C \norm{\conn u}$.
Thus, if we further assume that $u \in \nul(\conn)$
then we obtain that $u = 0$.
On letting 
$Z = \set{u \in \Lp{2}(\cM): \int_{\cM} u\ d\mu_\mgt = 0}$, 
that is precisely 
$$\set{0} = \nul(\Div_A) \intersect \Sob{1,2}(\cM) \intersect Z
	= \nul(\Div_A) \intersect Z,$$
since $\nul(\Div_A) \subset \Sob{1,2}(\cV)$. 
Therefore, $Z \subset \nul(\Div_A)^{\perp} = \close{\ran(\Div_A)}$.
For the reverse inclusion, let $u \in \close{\ran(\Div_A)}$.
Then, there exists a sequence $v_n \in \dom(\Div_A)$
such that $u = \lim_{n \to \infty} \Div_A v_n$ in $\Lp{2}(\cM)$.
Then,
\begin{multline*} 
\int_{\cM} u\ d\mu_{\mgt} 
	= \inprod{u, 1}
	= \lim_{n \to \infty} \inprod{\Div_A v_n, 1} \\
	= \lim_{n \to \infty} \inprod{-\divv_\mgt A \conn v_n, 1}
	= \lim_{n \to \infty} \inprod{A \conn v_n, \conn (1)}
	= 0.
\end{multline*}
\end{proof}

Next, we note that 
since $\Div_A$ preserves the spaces $\nul(\Div_A)$
and $\close{\ran(\Div_A)}$, we obtain that the 
restricted operator
$$\Div_A^R = \Div_A\rest{\close{\ran(\Div_A})}: 
	\close{\ran(\Div_A)} \to \close{\ran(\Div_A)},$$
and 
$$J_A^R = J_A\rest{\close{\ran(\Div_A)} \intersect \Sob{1,2}(\cM)},$$
with $\dom(J_A^R) = \Sob{1,2}(\cM) \intersect \close{\ran(\Div_A)}$.

\begin{proposition}
\label{Prop:Div_A^RDom}
The operator $\Div_A^R$ is a closed, densely-defined operator
with associated form $J_A^R$. 
\end{proposition}
\begin{proof}
By definition, we obtain that $\dom(\Div_A^R) = \dom(\Div_A) \intersect \close{\ran(\Div_A)}$.
Now, let $D_A^R$ be the operator given via the form 
$J_A^R$. We note that $J_A^R$ is both densely-defined
and closed.
Now, note that 
$$\dom(D_A^R) = \set{ u \in \close{\ran(\Div_A)}: 
	\Sob{1,2}(\cM) \intersect \ran(\Div_A) \ni v \mapsto J_A^R[u,v]\ \text{is continuous} }.$$
It is easy to see that $\dom(\Div_A) \intersect \close{\ran(\Div_A)} \subset \dom(D_A^R)$.

For the reverse inclusion, assume that $ u\in \dom(D_A^R)$, 
and $v \in \ran(\Div_A) \intersect \Sob{1,2}(\cM)$
so that $v\mapsto \inprod{A\conn u, \conn v}$ is continuous.
Now, we have that $\ran(\Div_A)$ is dense in $\Lp{2}(\cM)$
as well as in $\Sob{1,2}(\cM) = \dom(\sqrt{\Div_A})$ (by a
functional calculus argument). 
Thus, this continuity is valid for every $v \in \Sob{1,2}(\cM)$
and hence, $u \in \dom(\Div_A)$. Since 
we have by assumption that $u \in \close{\ran(\Div_A)}$,
we have that $\dom(D_A^R) \subset \dom(\Div_A) \intersect \close{\ran(\Div_A)}$.

It is easy to see that $\Div_A^R = D_A^R$ 
which is a closed, densely-defined operator 
by the Lax-Milgram theorem.
\end{proof}

We compute the spectrum of $\Div_A$ via the spectrum
for $\Div_A^R$.

\begin{proposition}
The spectra of the operators $\Div_A$ and $\Div_A^R$ relate 
by: 
$$\spec(\Div_A) = \set{0} \union \spec(\Div_A^R)
	\subset \set{0} \union [\kappa\lambda_1(\cM,\mgt), \infty).$$
\end{proposition}
\begin{proof}
Fix $\zeta \neq 0$ and $\zeta \in \rset(\Div_A^R)$,
That is, 
$\zeta - \Div_A^R: \dom(\L_A^R) = \close{\ran(\Div_A)} \intersect \dom(\Div_A)
	\to \close{\ran(\Div_A^R)}$
is invertible.
Thus, for any $u \in \close{\ran(\Div_A)} \intersect \dom(\Div_A)$,
there exists a $v \in \close{\ran(\Div_A)}$
such that $u = (\zeta - \Div_A^R)^{-1}v$ or
equivalently, $v = (\zeta - \Div_A^R)u$.
But we have that 
$\Div_A^R = \Div_A\rest{\close{\ran(\Div_A)} \intersect \dom(\Div_A)}$
and therefore, we obtain that
$v = (\zeta - \Div_A)u$.
That is, $u = (\zeta - \Div_A^R)^{-1}(\zeta - \Div_A)u$.
From this, we also obtain that $v = (\zeta - \Div_A)u = (\zeta - \Div_A)(\zeta - \Div_A^R)^{-1}v$.
That is, on $\close{\ran(\Div_A)}$, 
$(\zeta - \Div_A)^{-1} = (\zeta - \Div_A^R)^{-1}$.

Now, fix $u \in \nul(\Div_A)$. Then,
we have that $(\zeta - \Div_A)u = \zeta u$.
Since we assume $\zeta \neq 0$, we obtain that
$(\zeta - \Div_A^{-1})u = \zeta^{-1} u$.
This proves that 
$\rset(\Div_A^R)\setminus\set{0} \subset \rset(\Div_A)$
or $\spec(\Div_A) \subset \spec \rset{\Div_A^R} \union \set{0}$

Now, suppose that $\zeta \in \rset(\Div_A)$.
Then, for $u \in \close{\ran(\Div_A)}$ there exists
a $v \in \close{\ran(\Div_A)} \intersect \dom(\Div_A)$
such that $u =(\zeta - \Div_A)v$. But 
$(\zeta - \Div_A)v = (\zeta - \Div_A^R)v$. 
By the invertibility of $(\zeta - \Div_A)$ we
obtain the invertibility of $(\zeta - \Div_A^R)$.
Thus, $\rset(\Div_A) \subset \rset(\Div_A^R)$
and $\spec(\Div_A^R) \subset \spec(\Div_A)$.
Since we already know that $0 \in \spec(\Div_A)$,
we obtain that 
$\spec(\Div_A^R) \union \set{0} \subset \spec(\Div_A)$.

Next, note that since $\Div_A^R$
is self-adjoint,
$\spec(\Div_A^R) \subset \close{\nr(\Div_A^R)}$
where 
$$\nr(\Div_A^R) = \set{\inprod{\Div_A u, u}: u \in\dom(\Div_A^R),\ \norm{u} = 1},$$
is the numerical-range of $\Div_A^R$.
Moreover, via the \Poincare inequality,
we obtain that
$$J_A^R[u,u] \geq \kappa \norm{\conn u}^2 \geq\kappa \lambda_1(\cM,\mgt) \norm{u}^2$$
for $u \in \close{\ran(\Div_A)} \intersect \Sob{1,2}(\cM)$, and where
$\lambda_1(\cM,\mgt)$ is the first nonzero eigenvalue
for the smooth Laplacian $\Lap_{\mgt}$.
This shows that $\rset(\Div_A^R) \subset \C\setminus [\kappa \lambda_1(\cM,\mgt), \infty)$.
\end{proof}

Moreover,
we obtain that the operator $\Div_A$
has discrete spectrum.

\begin{proposition} 
The spectrum $\spec(\Div_A) = \set{0 < \lambda_1 \leq \lambda_2 < \dots}$
is discrete and  $\lambda_1 \geq \kappa\lambda_1(\cM,\mgt)$.
\end{proposition}
\begin{proof}
Fix $\delta > 0$ and write $\Div_{A,\delta}u = \Div_A u + \delta u$.
It is easy to see that 
$\spec(\Div_{A,\delta}) = \set{\delta} \union [\kappa\lambda_1(\cM,\mgt) + \delta, \infty)$.
Moreover, the operator $\Div_{A,\delta}$ is invertible, in fact,
$\Div_{A,\delta}^{-1}$ is a resolvent of $\Div_A$ and thus
$\Div_{A,\delta}^{-1}: \Lp{2}(\cM) \to \Lp{2}(\cM)$ boundedly.

Furthermore, note that
$$\norm{\conn \Div_{A,\delta}^{-1}u} 
	\lesssim \norm{\sqrt{\Div_{A, \delta}} \Div_{A,\delta}^{-1}u}
	= \norm{\Div_{A,\delta}^{-\frac{1}{2}}u}
	\lesssim \norm{u}.$$
Thus, we can obtain that $\Div_{A,\delta}^{-1}:\Lp{2}(\cM) \to \Sob{1,2}(\cM)$
boundedly. Let us call this operator $\Div_{A,\delta,\Sob{1,2}}^{-1}$.

Now, on a compact manifold, the inclusion map $E: \Sob{1,2}(\cM) \to \Lp{2}(\cM)$
is compact. Thus, we can write $\Div_{A,\delta}^{-1}  = E \Div_{A,\delta,\Sob{1,2}}^{-1}$.
This shows that $\Div_{A,\delta}^{-1}$ is compact. 
That is, $\spec(\Div_{A,\delta}^{-1})$ is discrete
with $0$ as the only accumulation point. Hence,
by combining with our previous bound for the
spectrum of $\spec(\Div_{A})$, we obtain the claim. 
\end{proof}

By the invertibility of $\Div_A^R$ on $\close{\ran(\Div_A)}$, 
and by our previous characterisation of 
$\close{\ran(\Div_A)}$, we
obtain the following first existence result.

\begin{proposition}
\label{Prop:EUDiv}
For every $f \in \Lp{2}(\cM)$ satisfying $\int_{\cM} f\ d\mu_\mgt = 0$,
we obtain a unique solution $u \in \Sob{1,2}(\cM)$ 
with $\int_{\cM} u\ d\mu_\mgt = 0$ to the equation
$\Div_A u = f$. This solution is given by
$u = (\Div_A^R)^{-1}f$.
\end{proposition}
\begin{proof}
The operator $\Div_A^R$ is invertible by the fact
that the associated form is bounded and coercive.
Moreover, it is easy to see that $\Div_A^{R}: \dom(\Div_A) \intersect \close{\ran(\Div_A)} \to \close{\ran(\Div_A)}$
and hence, $(\Div_A^{R})^{-1}: \close{\ran(\Div_A)} \to \dom(\Div_A) \intersect \close{\ran(\Div_A)}$.
The uniqueness is by virtue of the fact that
$\nul(\Div_A^R) = \set{0}$. The fact that
the solution $u \in \Sob{1,2}(\cM)$ easily 
follows since $(\dom(\Div_A), \norm{\mdot}_{\Div_A}) \subset \Sob{1,2}(\cM)$
is a continuous embedding.
\end{proof}

\subsection{Existence and uniqueness for a similar problem}

Let us return to the situation where 
$\mg$ is a rough metric on $\cM$. 
Recall that in this situation, there exists
a constant $C \geq 1$ and a smooth 
metric $\mgt$ that is $C$-close to $\mg$.
Let $B \in \Lp{\infty}(\Tensors[1,1]\cM)$ be
real-symmetric such that
$$ \mg_x(u,v) = \mgt_x(B(x)u,v)$$
for almost-every $x \in \cM$.
Let $\theta(x) = \sqrt{\det B(x)}$
for almost-every $x \in \cM$
so that $\mg(x) = \theta(x)d\mu_\mgt(x)$.

For the sake of convenience, 
we write $\Dir_{x} = -\divv_\mg \omega_x \conn$.
It is easy to see that 
$J_{x}[u,v] = \inprod{\omega_x \conn u, \conn v}_\mg$
is the real-symmetric form associated to $\Dir_x$.

First, we note the following.
\begin{proposition}
\label{Prop:Uni1}
There exist $\kappa > 0$ such that 
$J_{x}[u,u]  \geq \kappa \norm{\conn u}^2_\mg$
uniformly for $x \in \cM$.
Moreover, $J_{x}[u,v] = \inprod{\omega_x B \theta \conn u, \conn v}_\mgt$
and $J_{x}[u,u] \geq \kappa C^{1+ \frac{n}{4}} \norm{\conn u}_\mgt$.
A function $\phi  \in \Sob{1,2}(\cM)$ solves
\eqref{Def:F} if and only if
$$ -\divv_\mgt(\theta B \omega_x \conn\phi) = \theta \eta.$$
\end{proposition}
\begin{proof}
First, we note that 
$\omega(x,y) > 0$ for all $x,y$. 
Secondly, since we assume that $\omega \in \Ck{0,1}(\cM^2)$,
for $k \geq 1$, by the virtue of compactness of $\cM$, 
we obtain that $\inf_{x,y} \omega(x,y) = \min_{x,y} \omega(x,y) > 0$. 
That is, set $\kappa = \min_{x,y} \omega(x,y) $ and we're done.

The description of $J_{x}$ in $\mgt$ and its
ellipticity estimate in $\mgt$ follow from 
Proposition \ref{Prop:RoughP} (i) with $p = 2$, and $r = 1$, $s = 0$. 

The equivalence of solutions for the \eqref{Def:F} 
simply follows from Proposition \ref{Prop:RoughP} (vi).
\end{proof}

This is the crucial observation which allows us to 
reduce the \eqref{Def:E} to solving a divergence 
form equation with bounded, measurable coefficients 
on the nearby smooth metric $\mgt$.

With the aid of this, we demonstrate
existence and uniqueness of solutions to
\eqref{Def:E}.

\begin{proposition}
\label{Prop:EUF}
Let $\eta \in \Lp{2}(\cM)$ satisfy $\int_{\cM} \eta\ d\mu_\mg = 0$. 
Then, there exists a unique solution $\phi \in \Sob{1,2}(\cM)$
satisfying \eqref{Def:F} such that
$\int_{\cM} \phi\ d\mu_\mg = 0$.
\end{proposition} 
\begin{proof}
Let $f = \theta \eta$. Then, note that 
$$ \int_{\cM} f\ d\mu_\mgt = \int_{\cM} \eta \theta\ d\mu_\mgt = \int_{\cM} \eta\ d\mu_\mg = 0.$$
Set $A = \theta B \omega_x$
and by what we have proved about $f$, 
we are able to apply 
Proposition \eqref{Prop:EUDiv}
to the operator $\Div_A$ in $\inprod{\mdot,\mdot}_\mgt$
to obtain a unique solution $\tilde{\phi}$ satisfying
$\Div_A \tilde{\phi} = f = \theta \eta$
with $\int_{\cM} \tilde{\phi}\ d\mu_\mgt = 0$.

Define 
$\phi(y) = \tilde{\phi}(y) - \fint_{\cM} \tilde{\phi(y)}\ d\mu_{\mg}$
which satisfies $J[\phi, f] = \inprod{\eta, f}_\mgt$
and we also find that
$$\int_{\cM} \phi(y)\ d\mu_\mg(y) 
	= \int_{\cM} \tilde{\phi}(y)\ d\mu(y) - \int_{\cM} (\fint_{\cM} \tilde{\phi}(y)\ d\mu_\mg(y))\ d\mu_\mg(y)
	= 0.$$ 
Thus, $\phi$ solves \eqref{Def:F}.

To prove uniqueness, let us fix two solutions 
$\phi^1$ and $\phi^2$ solving
\eqref{Def:F} with $\int_{\cM} \phi^i\ d\mu_\mg = 0$. 
Then, on writing $\psi = \phi^1 - \psi^2$, we
obtain that $\psi$ satisfies
$$ -\divv_\mg \omega_x \conn \psi = 0$$
with $\int_{\cM} \psi\ d\mu = 0$.
Now, define $\tilde{\psi}(y) = \psi(y) - \fint_{\cM} \psi d\mu_\mgt$.
It is easy to see that 
$\tilde{\psi}$
satisfies
$$-\Div_A \tilde{\psi} = 0.$$ 
with $\int_{\cM} \tilde{\psi}\ d\mu_\mgt = 0$. 
Thus, by the uniqueness guaranteed by Proposition \ref{Prop:EUDiv},
we obtain that $\tilde{\psi} = 0$.
That is, 
$\psi(y) = \fint_{\cM} \psi(y)\ d\mu_\mgt(y)$, and on 
integrating this with respect to $\mg$, we obtain that
$$ 0 = \mu_\mg(\cM) \fint_{\cM} \psi(y)\ d\mu_\mgt(y).$$
That is, $\psi = \tilde{\psi} = 0$.
\end{proof}

\subsection{Operator theory of $x \mapsto \Dir_x$}

Let us return to the PDE 
\eqref{Def:F}, and recall the
operator $\Dir_x = -\divv_\mg \omega_x \conn$
where $(x,y) \mapsto \omega_x(y) \in \Ck{0,1}(\cM^2)$,
and $x \mapsto \omega_x \in \Ck{k,\alpha}(\cN)$. 
Let $J_x[u,v]$ be its associated
real-symmetric form,
$J_x[u,v] = \inprod{\omega_x \conn u,\conn v}.$

In order to understand the regularity
$x \mapsto \phi_{t,x,v}$ of solutions
to \eqref{Def:E}, we need to 
prove some preliminary regularity
results about the operator family $\Dir_x$.
First, we obtain the constancy of
domain as well as the following 
formula. 

\begin{proposition}
\label{Prop:ConstDom}
The family of operators $\cM \ni x \mapsto \Dir_x$
satisfies $\dom(\Dir_x) = \dom(\Lap_\mg)$ and 
$\Dir_x u = \omega_x \Lap_\mg u - \mg(\conn u, \conn \omega_x).$
\end{proposition}
\begin{proof}
Fix $x \in \cM$ and $u, v \in \Sob{1,2}(\cM)$. Then, note that $\omega_x v \in \Sob{1,2}(\cM)$
and that,
$$
\inprod{\conn u, \conn(\omega_x v)}_\mg 
	= \inprod{\conn u, \omega_x \conn v + v\conn \omega_x}_\mg
	= \inprod{\omega_x \conn u, \conn v}_\mg + \inprod{\mg(\conn u, \conn \omega_x), v}_\mg.$$
That is,
$
\inprod{\omega_x \conn u, \conn v}_\mg = \inprod{\conn u, \conn(\omega_x v)}_\mg  
	- \inprod{\mg(\conn u, \conn \omega_x), v}_\mg.$

First we show that for any $u \in \Sob{1,2}(\cM)$,
$v \mapsto \inprod{\mg(\conn u, \conn \omega_x), v}_\mg$ is continuous. 
Observe that
$
\modulus{\inprod{\mg(\conn u, \conn \omega_x),v}_\mg} \leq 
	\norm{\mg(\conn u,\conn \omega_x)}_\mg \norm{v}_\mg$
by the Cauchy-Schwarz inequality. Moreover,
$$
\norm{\mg(\conn u,\conn\omega_x)}_\mg^2 
	= \int_{\cM} \modulus{\mg(\conn u, \conn \omega_x)}^2\ d\mu_\mg
	\leq \int_{\cM} \modulus{\conn u}^2 \modulus{\conn \omega_x}^2\ d\mu_\mg.$$
However, since $y \mapsto \omega_x(y) \in \Ck{0,1}(\cM)$ and $\cM$ is compact,
we have that $\esssup_{y} \modulus{\conn \omega_x (y)} \leq C$, and hence,
$\norm{\mg(\conn u,\conn \omega_x)}_\mg \leq  C \norm{\conn u}_\mg.$
This proves that
$v \mapsto \inprod{\mg(\conn u, \conn \omega_x), v}_\mg$ is continuous.

Now, suppose that $u \in \dom(\Lap_\mg)$, then 
$\inprod{\conn u, \conn(\omega_x v)}_\mg = \inprod{\Lap_\mg u, \omega_x v}$
and hence, $v \mapsto \inprod{\conn u, \conn(\omega_x v)}_\mg$
is continuous. Since we have already shown that
$v \mapsto \inprod{\mg(\conn u, \conn \omega_x), v}_\mg$ is continuous,
we obtain that $v \mapsto \inprod{\omega_x \conn u, \conn v}$ is continuous.
Hence, $u \in \dom(\Dir_x)$ which
proves that $\dom(\Lap_\mg) \subset \dom(\Dir_x)$.

Similarly, for $u \in \dom(\Dir_x)$, we find that
$v \mapsto \inprod{\conn u, \conn(\omega_x v)}_\mg$ 
is continuous. Hence, $u \in \dom(\omega_x \Lap_\mg) = \dom(\Lap_\mg)$.
This shows that $\dom(\Dir_x) \subset \dom(\Lap_\mg)$.
\end{proof}

\begin{remark}
We note that an immediate consequence
of this is that the unique solution $\phi$
to \eqref{Def:F} satisfies $\phi \in \dom(\Lap_\mg)$
as  $\dom(\Lap_\mg) = \dom(\Dir_x) = \dom(\Div_{x})$.
This observation is essential in our approach to regularity. 
\end{remark}

We also obtain the following uniform boundedness 
for the operator family parametrised in $x \in \cN$.

\begin{proposition}
\label{Prop:UniBdd}
The family of operators 
$\cM \ni x \mapsto \Dir_x: (\dom(\Lap_\mg),\norm{\mdot}_{\Lap_\mg}) \to \Lp{2}(\cM)$
is a uniformly bounded family of operators.
Moreover,
$\norm{u}_{\Dir_x} \simeq \norm{u}_{\Lap_{\mg}}$
holds with the implicit constant independent of $x \in \cM $.
\end{proposition}
\begin{proof}
We show that 
$\norm{\Dir_x u} \lesssim \norm{u}_{\Lap_\mg} = \norm{\Lap_\mg u} + \norm{u}$,
where the implicit constant is independent of $x \in \cM$.
Fix $x \in \cM$, and note that
\begin{multline*}
\norm{\Dir_x u} 
	\leq \norm{\omega_x \Lap_\mg u} + \norm{\mg(\conn u, \conn \omega_x)} \\
	\leq \cbrac{\sup_{y\in \cM } \modulus{\omega_x(y)}} \norm{\Lap_\mg u}
		+ \cbrac{\esssup_{y\in \cM } \modulus{\conn \omega_x(y)}} \norm{\conn u}.
\end{multline*}
But since $(x,y)\mapsto \omega_x(y) \in \Ck{0,1}(\cM^2)$, 
the Lipschitz constant is in both variables, and hence, 
there is a $C > 0$ such that 
$ \esssup_{y \in \cM} \modulus{\conn \omega_x(y)} \leq C.$
The quantity $\sup_{y \in \cM} \modulus{\omega_x(y)}$ is
also independent of $C$ simply by coupling the
continuity of $(x,y)\mapsto \omega_x(y)$
with the compactness of $\cM$.

Now, note that by ellipticity, 
$$\norm{\conn u}^2 
	\leq \modulus{\inprod{\Lap_\mg u, u}} 
	\leq \norm{\Lap_\mg u}\norm{u}
	\leq \norm{\Lap_\mg u}^2 + \norm{u}^2.$$
This complete the proof.
The reverse inequality is argued similarly
on noting that $\omega_x(y) > 0$ for all $x, y \in \cM$.
\end{proof}

\begin{remark}
Note that the two previous propositions
are valid on all of $\cM$, not just on
$\cN \subset \cM$ where $x \mapsto \omega_x$
enjoys higher regularity.
\end{remark}

Let $v  \in \tanb_x \cM$ and $\gamma:(-\epsilon, \epsilon) \to \cM$
such that $\gamma(0) = x$ and $\dot{\gamma}(0) = v$.
Let $f: \cN \to \cV$, 
where $\cV$ where $\cV$ is some normed vector
space. Then, we write the difference  quotient as  
$$ Q^v_sf(x) = \frac{f(x) - f(\gamma(s))}{s}.$$
We define the 
\emph{directional derivative} of $f$ (when it exists and it
is independent of the generating curve $\gamma$) to be  
$$(\extd_xf(x))(v) = \lim_{s \to 0} Q^v_sf(x).$$ 

In our particular setting, we consider $\cV = \Lp{2}(\cM)$
with the weak topology for the choice $f(x) = \Dir_x$. 
More precisely, if there exists $\tilde{\Dir}_x: \dom(\Lap_\mg) \to \Lp{2}(\cM)$
satisfying 
$$\lim_{s \to 0} \inprod{Q^v_s \Dir_x u, w} = \inprod{\tilde{\Dir}_x u, w},$$
for every $w \in \Sob{1,2}(\cM)$, we
say that $\Dir_x$ has a \emph{(weak) derivative}
at $x$ and write  
$(\extd_x\Dir_x) = \tilde{\Dir}_x$.
In what is to follow, we will see that this is
a sufficiently strong enough notion of derivative 
to obtain regularity properties for the 
flow defined by \eqref{Def:GM}.

With this notation at hand, we prove the following
important proposition.

\begin{proposition}
\label{Prop:OpDiff}
The operator valued function
 $\cN \ni x \mapsto \Dir_x: \dom(\Lap_\mg) \to \Lp{2}(\cM)$ 
is weakly differentiable $k$ times.
At each $x \in \cN$ and for every $v \in \tanb_x \cM$,
$$(\extd_x\Dir_x)(v) = -\divv_\mg ((\extd_x\omega_x)(v)) \conn: \dom(\Lap_\mg) \to \Lp{2}(\cM)$$
is densely-defined and symmetric. Moreover,
inside a chart $\Omega \Subset \cN$ containing $x$
for which the vector $v$ is constant, 
there is a constant $C_{\Omega}$ such that 
$$\norm{(\extd_x \Dir_x)(v) u} \leq C_{\Omega} \norm{u}_{\Lap_\mg}.$$
\end{proposition}
\begin{proof}
Fix $x \in \cN$ and a chart $\Omega \Subset \cN$
with a constant vector $v \in \tanb_x \cM$, 
and note that because of
the higher regularity of $\omega_x$ at $x$, 
 we have   
$(x,y) \mapsto (\extd_x \omega_x(y))(v) \in \Ck{0}(\Omega \times \cM)$
and $x \mapsto (\extd_x \omega_x(y))(v) \in \Ck{0,1}(\cM)$.
Coupling this with the compactness of $\cM$ and $\close{\Omega}$,
there exists $\Lambda > 0$ such that
$-\Lambda < (\extd_x \omega_x(y))(v) < \Lambda$,
for all $x \in \Omega$ and $y \in \cM$.
Thus, let $f_{x,\epsilon} = (\extd_x \omega_x)(v) + \Lambda + \epsilon$
and define 
 $K_{\epsilon}[u,w] = \inprod{ f_{x,\epsilon} \conn u, \conn w}$.
By the Lax-Milgram theorem, the operator
associated to the form $K_\epsilon$ is exactly 
$$\tilde{\Dir}_{x,\epsilon} = -\divv_\mg [((\extd_x\omega_x)(v)) 
	+ (\Lambda + \epsilon)] \conn,$$ and is guaranteed to be 
non-negative self-adjoint. Since we have that $\omega_x$ is $k$
times differentiable at our chosen $x$,
the map $y \mapsto f_{x,\epsilon}(y)$ is still Lipschitz
and hence, we are able to apply Proposition \ref{Prop:ConstDom}
with $f_{x,\epsilon}$ in place of $\omega_x$
to obtain that $\dom(\tilde{\Dir}_{x,\epsilon}) = \dom(\Lap_\mg).$
Consequently, we obtain that 
$\tilde{\Dir}_{x,\epsilon} - (\Lambda + \epsilon)\Lap_\mg$
has domain $\dom(\Lap_\mg)$ and so, an easy calculation
via the defining form, demonstrates that 
$$\Dir'_x u = -\divv_{\mg} (\extd_x\omega_x)(v) \conn u 
= \tilde{\Dir}_{x,\epsilon}u - (\Lambda + \epsilon)\Lap_\mg u,$$
from which its clear that the operator is densely-defined.

A repetition of the argument in Proposition \ref{Prop:UniBdd},
 utilising the higher regularity of 
$x \mapsto \omega_x$ on $\cN$,
there exists a constant $C_\Omega > 0$ such that 
$\modulus{(\extd_x \omega_x(y))(v)} \leq C_{\Omega}$
for all $x \in \Omega$ and almost-every $y \in \cM$.
Thus,  $\norm{\Dir_{x,\epsilon} u} \leq C_{\Omega} \norm{u}_{\Lap_\mg}$
and the estimate in the conclusion follows.

Now we show that the formula in the conclusion
is valid.
Fix $w \in \Sob{1,2}(\cM)$, $u \in \dom(\Lap_\mg)$
and compute
\begin{multline*} 
\lim_{s \to 0} \inprod{Q^v_s\Dir_x u, w}
	= \lim_{s \to 0} \inprod{ -\divv Q_s^v\omega_x\conn u, w}
	= \lim_{s \to 0} \inprod{ Q_s^v \omega_x \conn u, \conn w} \\
	= \lim_{s \to 0} \int_{\cM} Q_s^v \omega_x(y) \mg_y(\conn u(y), \conn w(y))\ d\mu_\mg(y).
\end{multline*}
Now, note that 
$$ \modulus{Q_s^v\omega_x(y)} \leq \lmodulus{\frac{\omega_x(y) - \omega_{\gamma(s)}(y)}{s}} \leq C$$
since $(x,y) \mapsto \omega_x(y) \in \Ck{0,1}(\cM^2)$
and $x \mapsto \omega_x \in \Ck{1}(\cM)$. So, 
we are able to apply the dominated convergence theorem
to obtain
\begin{multline*}
\lim_{s \to 0} \inprod{Q^v_s\Dir_x u, w}
	= \int_{\cM} \lim_{s \to 0} Q^s_v\omega_x(y) \mg_y(\conn u(y), \conn w(y))\ d\mu_\mg(y)  \\
	= \inprod{(\extd_x \Dir_x)(v)\conn u, \conn w} 
	= \inprod{-\divv_\mg (\extd_x \Dir_x)(v) \conn u, w},
\end{multline*}
where the last equality follows
from the fact that we assume that $u \in \dom(\Lap_\mg)$
and we have already shown  that $\dom(-\divv_\mg ((\extd_x\omega_x)(v)) \conn) = \dom(\Lap_\mg).$

The equality of operators in the conclusion follows from the fact
that $w \in \Ck[c]{\infty}(\cM)$ is dense in $\Lp{2}(\cM)$.
\end{proof}

\begin{remark}
Let 
$v_1, \dots, v_l \in \tanb_x \cM$,
with $k \leq l$, and note that the map
$(x,y) \mapsto (\extd_x^l \omega_x)(v_1, \dots, v_l)
\in \Ck{0,1}(\cM^2)$, where
$(\extd_x^2\omega_x)(v_1,v_2) = (\extd_x (\extd_x\omega_x)(v_1))(v_2)$. 
Thus, on applying this proposition repeatedly,
we can assert that
$$(\extd_x^l\Dir_x)(v_1,\dots,v_l) = 
	-\divv_\mg ((\extd_x^l\omega_x)(v_1,\dots,v_l)) \conn: \dom(\Lap_\mg) \to \Lp{2}(\cM)$$
is a densely-defined operator.
\end{remark}

Now, we are able to prove the following
product rule for the operator
$\Dir_x$. This product rule is the
essential tool
for obtaining the existence of weak derivatives
$\partial_x \phi_x$ of solutions for \eqref{Def:E}.

\begin{proposition}
Let $x \mapsto u_x: \cN \to \dom(\Lap_\mg)$, $v \in \tanb_x \cM$
and suppose that $(\extd_x u_x)(v)$ exists weakly.
Then  $(\extd_x\Dir_x u_x)(v)$ exists weakly
if and only if $\Dir_x((\extd_x u_x)(v))$ exists weakly
and 
$$ (\extd_x\Dir_x u_x)(v) = (\extd_x\Dir_x)(v) u_x + \Dir_x((\extd_x u_x)(v)).$$
\end{proposition}
\begin{proof}
Fix $w \in \Sob{1,2}(\cM)$ and define
$f(x,y) = \inprod{\Dir_x u_y , w}$.
Let $\gamma:(-\epsilon,\epsilon) \to \cM$ be a curve in $\cM$
satisfying $\gamma(0) = x$ and $\dot{\gamma}(0) = v$.
Now, note that for $s > 0$, 
\begin{equation}
\tag{$\dagger$}
\label{Eq:Lim} 
\frac{1}{s}[f(\gamma(s),\gamma(s)) - f(x,x)] 
	= \frac{1}{s}[f(\gamma(s),\gamma(s)) - f(x,\gamma(s))]  +
		\frac{1}{s}[f(x,\gamma(s)) - f(x,x)]
\end{equation}
By Proposition \ref{Prop:OpDiff}, 
$(\extd_xf(x,y))(v)\rest{y = x}$ exists and 
\begin{align*}
(\extd_xf(x,y))(v)\rest{y = x}
	&= \lim_{s \to 0} \frac{1}{s}[f(\gamma(s), y) - f(x,y)]\rest{y = x}\\
	&= \lim_{t \to 0} \lim_{s \to 0}\frac{1}{s}[f(\gamma(s),\gamma(t)) - f(x,\gamma(t))] \\
	&= \lim_{s \to 0} \frac{1}{s}[f(\gamma(s),\gamma(s)) - f(x,\gamma(s))].
\end{align*}

Assume that 
$ (\extd_x\Dir_x u_x)(v)$ exists
weakly. Then, by \eqref{Eq:Lim}, 
$\lim_{s \to 0} s^{-1} [f(x,\gamma(s)) - f(x,x)]$ exists
and so, by further choosing $w \in \dom(\Lap_\mg)$, 
\begin{multline*}
\lim_{s \to 0}\frac{1}{s}[f(x,\gamma(s)) - f(x,x)] 
= \lim_{s \to 0} \inprod{\frac{1}{s} \Dir_x (u_{x} - u_{\gamma(s)}), w}   \\
= \lim_{s \to 0} \inprod{Q_s^vu_x, \Dir_x w} 
= \inprod{(\extd_x u_x)(v),\Dir_x w},
\end{multline*}
Also, by \eqref{Eq:Lim},
$$\inprod{(\extd_x u_x)(v),\Dir_x w}
	= \inprod{\partial_x(\Dir_x u_x)(v) - (\partial_x\Dir_x)(v) u_x, w},$$
and since the right hand side is continuous in $w$,
we obtain that $(\extd_x u_x)(v) \in \dom(\Dir_x) = \dom(\Lap_\mg)$.

Now, if $\Dir_x(\extd_x u_x)(v)$ exists
weakly, then from \eqref{Eq:Lim},
we are able to assert that
the limit 
$\lim_{s \to 0} s^{-1} [f(\gamma(s),\gamma(s)) - f(x,x)]$
exists, 
which is precisely that
$\extd_x(\Dir_x u_x)(v)$ exists weakly.
The product rule formula is obvious from these computations.
\end{proof}

\begin{remark}
If the function $(x,y) \mapsto \omega_x(y) \in \Ck{k}(\cM^2)$
for $k \geq 1$, then we are able to 
perform this analysis in the uniform operator
topology 
$\bddlf( (\dom(\Lap_\mg), \norm{\mdot}_{\Lap_\mg}), \Lp{2}(\cM))$.
This involves estimating the term 
$\sup_{x,y\in \cM} \modulus{\conn Q_s^vw_x(y)}$
and showing that this quantity tends to $0$
as $s \to 0$. It is clear 
that such an estimate cannot be made
even with the supremum replaced by an
essential supremum when $(x,y) \mapsto \omega_x(y)$
is only Lipschitz.
\end{remark}

\subsection{Regularity of solutions}

We combine the results obtained 
in the previous subsections to prove the following
regularity theorem for solutions
to \eqref{Def:F}. Recalling that $\Div_x = -\divv_{\mgt}(\theta B \omega_x)\conn$, 
we note the following lemma.
 
\begin{lemma}
\label{Lem:Uni}
Let $\int_{\cM} u\ d\mu_\mg = 0$. Then,
$\norm{\Div^{-\frac{1}{2}}_x u} \lesssim \norm{u}$
and  
$\norm{\Div_{x}^{-1}u} \lesssim \norm{u}$, 
where the implicit constants are independent of $x$.
\end{lemma} 
\begin{proof}
For this, note that
$\Dir_{x} = \theta^{-1} \Div_{x}$ and that
$$\inprod{\Div_{x} v, v}_{\mgt} \geq \kappa \norm{\conn v}_{\mgt}.$$
If we assume that $\int_{\cM} v\ d\mu_\mgt = 0$, 
then we have by Proposition \ref{Prop:Uni1} that 
$\norm{\sqrt{\Div_{x}} v} \geq \kappa \lambda_1(\cM,\mgt)\norm{v}$
uniformly in $x$. This shows the uniform boundedness 
for $\Div^{-\frac{1}{2}}_x$.  

Next, set $v = \sqrt{\Div_{x}}w$
for $w \in \dom(\Lap_\mg )$, we obtain that 
$\norm{\Div_{x}w}_\mgt \gtrsim \kappa^2 \lambda_1(\cM,\mgt)^2 \norm{w}_\mgt.$
That is, $\norm{\Div_{x}^{-1} w}_\mgt \lesssim \norm{w}_\mgt$.
Now, for $u \in \dom(\Lap_\mg)$ satisfying 
$\int_{\cM} u\ d\mu_\mg = 0$, we have that
$w = \theta u$ satisfies $\int_{\cM} w\ d\mu_\mgt = 0$
and $\Div_{x}^{-1}w = \Div_{x}^{-1}u$ and so 
$\norm{\Div_x^{-1}u}_{\mg} 
	\simeq \norm{\Div_{x}^{-1}w}_{\mgt} 
	\lesssim \norm{w}_{\mgt} = \norm{\theta u}_{\mgt} = \norm{u}_{\mg}.$
\end{proof}

With this tool in hand, we present the following
regularity theorem.

\begin{theorem}
\label{Thm:Reg}
Suppose that $k \geq 1$ and $(x,y) \mapsto \omega_x(y) \in \Ck{0,1}(\cM^2)$
and $x\mapsto \omega_x \in \Ck{k}(\cN)$. 
Moreover, suppose that $(x,y) \mapsto \eta_x(y) \in \Ck{0}(\cN \times \cM)$
and $x \mapsto \eta_x(y) \in \Ck{l}(\cN)$ where $l \geq 1$.
If at  $x \in \cN$,
$\phi_x$ solves \eqref{Def:F} 
with $\int_{\cM} \phi_x\ d\mu_\mg = \int_{\cM} \eta_x\ d\mu_\mg = 0$,
the map $x \mapsto \inprod{\eta_x, \phi_x} \in \Ck{\min\set{k,l}-1,1}(\cN)$.
\end{theorem}
\begin{proof}
Fix $v \in \tanb_x \cM$ and assuming
$(\extd_x \phi_x)(v)$ exists in $\dom(\Lap_\mg)$,
we obtain that 
$(\extd_x \eta_x)(v) = (\extd_x \Dir_x)(v)\phi_x + \Dir_x (\extd_x \phi_x)(v).$
Thus, on writing 
$$\eta'_{x,v} =  (\extd_x \eta_x)(v) - (\extd_x \Dir_x)(v)\phi_x$$
and 
rearranging the previous expression, we
note that $(\extd_x \phi_x)(v)$ exists
as a solution to 
$
\Dir_x (\extd_x \phi_x)(v) = \eta'_x$.
We further note that this PDE is again of the
form \eqref{Def:F}.

Fix a curve $\gamma:(-\epsilon,\epsilon) \to \cM$ such that
$\gamma(0) = x$ and $\dot{\gamma}(0) = v$. Then,
$$
\int_{\cM} (\extd_x \eta_x)(v)\ d\mu_\mg 
	= \int_{\cM} \ddt\rest{t = 0} \eta_{\gamma(t)}\ d\mu_\mg
	= \ddt\rest{t = 0} \int_{\cM} \eta_{\gamma(t)}\ d\mu_\mg = 0$$
simply by virtue of the
fact that  $\int_{\cM} \eta_x\ d\mu_\mg = 0$ for each $x$.

Next, note that by Proposition \ref{Prop:OpDiff}
$$
\int_{\cM} (\extd_x\Dir_x)(v)\phi_x\ d\mu_\mg
	= \inprod{(\extd_x\Dir_x)(v)\phi_x, 1}_\mg
	= \inprod{(\extd_x \omega_x)(v) \conn \phi_x, \conn(1)}_\mg
	= 0.$$
Thus, we have shown that $\int_{\cM} \eta'_x\ d\mu_\mg = 0$
and hence, by Proposition \ref{Prop:EUF}, we obtain 
that $(\extd_x \phi_x)(v) \in \dom(\Lap_\mg) \subset  \Sob{1,2}(\cM)$ exists.

Next, we show that $x \mapsto \inprod{\eta_x, \phi_x}$ is
differentiable. For that, let us write
$f(x,y) = \inprod{\eta_x, \phi_y}$. By  \eqref{Eq:Lim}, 
we have that
$$ (\extd_xf(x,x))(v) = (\extd_xf(x,y))(v)\rest{y = x} +
	\extd_x(f(y,x))(v)\rest{y = x}$$
when the limits exist.

So, first for the first
expression on the left hand side, 
$$
(\extd_xf(x,y))(v)\rest{y = x}
	= \lim_{s \to 0} \inprod{Q^v_s \eta_x, \phi_x}
	= \inprod{(\extd_x \eta_x)(v), \phi_x}.$$
For the second expression, 
$$
(\extd_xf(y,x))\rest{y = x})
	= \inprod{ \eta_x, (\extd_x \phi_x)(v)}.$$

Now we show that the directional derivative 
is bounded in small neighbourhoods
containing $x$.  So, fix $\Omega \Subset \cN$
a coordinate chart containing $x$ in which 
the vector $v$ is constant in this chart. 
We note that it suffices to show that 
$$\modulus{(\extd_x \inprod{\eta_x, \phi_x})(v)} \lesssim p(\norm{\eta_x}, \norm{(\extd_x \eta_x)(v)}),$$
for a polynomial $p$ since $\norm{\eta_x} \leq C$ and $\norm{(\extd_x \eta_x)(v)}  \leq C$,
where the constant $C$ and the implicit constant
depend on $\Omega$.
This demonstrates that 
$x \mapsto \inprod{\eta_x, \phi_x}$ is continuous
at $x$ with bounded derivatives, 
which in turn implies that this function 
is Lipschitz, and moreover that the differential
exists almost-everywhere.

Recall that by Proposition \ref{Prop:EUF}, 
 $\phi_x = \Div_{x}^{-1}\theta\eta_x + c$,
where $c_x = \fint_{\cM} \Dir_{x}^{-1}\theta\eta_x\ d\mu_\mg$. 
Similarly, $(\extd_x \phi_x)(v) = \Div_{x}^{-1} \theta \eta_{x,v}' + c'$
where $c_x' = \fint_{\cM} \Div_{x}^{-1} \theta\eta_{x,v}'\ d\mu_\mg$.
Hence, 
$$
\modulus{(\extd_xf(x,y))(v)\rest{y = x}} 
	= \modulus{\inprod{ \eta_x, \Div_{x}^{-1}\theta\eta_x + c_x}} 
	\leq \norm{\eta_x} \norm{\Div_x^{-1} \theta\eta_x} + \norm{c_x} \norm{\theta \eta_x}.
$$
Now, note that by Lemma \ref{Lem:Uni},
$\norm{\Div_x^{-1} \theta \eta_x} \lesssim \norm{\eta_x}$, 
where the constant is uniform in $x \in \cM$, and that
$$
\fint_{\cM} \Div_x^{-1} \theta \eta_x\ d\mu_\mg 
	\lesssim \norm{\Div_x^{-1} \theta \eta_x} \lesssim \norm{\eta_x},$$	
which shows that $\norm{c_x} \lesssim \norm{\eta_x}$.
Thus, 
$$\modulus{(\extd_xf(x,y))(v)\rest{y = x}} \lesssim \norm{\eta_x}^2$$
with the constant independent of $x \in \cM$.

We estimate the remaining term,
\begin{align*}
\modulus{(\extd_xf(y,x))\rest{y = x})}
	&= \modulus{\inprod{\eta_x, \Div_{x}^{-1}\theta (\extd_x \eta_x)(v) - \Div_x^{-1} \theta (\extd_x \Div_x)(v)\phi_x) + c_x'}} \\
	&\leq \modulus{\inprod{\eta_x, \Div_{x}^{-1}\theta(\extd_x \eta_x)(v)}} +
		\modulus{\inprod{\eta_x, \Div_x^{-1} \theta (\extd_x \Div_x)(v)\phi_x)}} + \modulus{\inprod{c_x', \eta_x}}.
\end{align*}
Now,
$$\modulus{\inprod{\eta_x, \Div_{x}^{-1}\theta (\extd_x \eta_x)(v)}}
	\lesssim \norm{\eta_x} + \norm{\Div_x^{-1} \theta (\extd_x \eta_x)(v)}
	\lesssim \norm{\eta_x}  + \norm{(\extd_x\eta_x)(v)},$$
and   
\begin{multline*}
\modulus{\inprod{\eta_x, \Div_x^{-1}\theta (\extd_x \Dir_x)(v)[\Div_{x}^{-1}\theta\eta_x + c]}}
	\lesssim  \norm{\eta_x} + \norm{\Div_x^{-1}\theta (\extd_x \Dir_x)(v)\Div_{x}^{-1}\theta\eta_x} \\
	\lesssim \norm{\eta_x} + \norm{(\extd_x \Dir_x)(v)\Div_{x}^{-1}\theta\eta_x},
\end{multline*}  
again by Lemma \eqref{Lem:Uni} where
the implicit constant is independent of $x$ 
since $(\extd_x \Dir_x)(v)c_x = 0$ by the fact that $c_x \in \nul(\conn)$. 
For the last term, note that
$$\norm{c'_x} \lesssim \norm{\eta'_{x,v}}
	\leq \norm{(\extd_x\eta_x)(v)} + \norm{(\extd_x \Dir_x)\Div_x^{-1}\theta \eta_x}.$$

By these calculations, it suffices to
show that $\norm{(\extd_x\Dir_x)(v)\Div_x^{-1} \theta \eta_x} \lesssim \norm{\eta_x}$,
where the implicit constant depends on $\Omega$.
In order to estimate this term, note that  
by Proposition \ref{Prop:OpDiff}
$\norm{(\extd_x\Dir_x)(v) u} \lesssim \norm{\Lap_\mg u} + \norm{u}$
uniformly in $x \in \Omega$,
and therefore, 
\begin{multline*}
\norm{(\extd_x \Dir_x)(v)\Div_{x}^{-1}\theta \eta_x} 
	\lesssim \norm{\Lap_\mg \Div_x^{-1} \theta \eta_x} + \norm{\Div_x^{-1} \theta \eta_x} \\
	\lesssim  \norm{\eta_x} + \norm{\Div_x^{-1} \theta \eta_x} + \norm{\Div_x^{-1}\theta\eta_x}
	\lesssim \norm{\eta_x}.
\end{multline*}

To prove higher differentiability and continuity
for $x \in \cN$, 
it suffices to repeat the argument
upon replacing $\eta_{x,v}'$ and $\omega_x$,
\emph{mutatis mutandis}, to solve for higher
weak derivatives. It is easy to see that 
this procedure can only be repeated
as many times as the minimum of the regularity of $\eta_x$ and
$\omega_x$.
\end{proof}

\section{The flow for rough metrics with Lipschitz kernels}
\label{Sec:Flow}

 In this section, we return back to the study of the main
flow problem \eqref{Def:GM}. For the benefit of the reader,
we recall the governing equations \eqref{Def:E} and \eqref{Def:GM}.

Fix $t > 0$, $x \in \cN$, and $v \in \tanb_x \cM$, 
and recall the following
linear PDE satisfying, for $\phi_{t,x,v} \in \Sob{1,2}(\cM)$,
\begin{equation}
\tag{CE}
\begin{aligned}
&-\divv_\mg ( \hk^\mg_t(x,y) \conn \phi_{t,x,v}(y)) =( \extd_x\hk^\mg_t(x,y))(v) \\
&\int_{\cM} \phi_{t,x,v}(y)\ d\mu_\mg(y) = 0.
\end{aligned} 
\end{equation}
The flow of Gigli-Mantegazza defined 
in \cite{GM} is then given by 
\begin{equation} 
\tag{GM}
\mg_t(u,v)(x) = \int_{\cM} \mg(\conn \phi_{t,x,u}(y), 
	\conn \phi_{t,x,v}(y))\ \hk^\mg_t(x,y)\ d\mu_\mg(y).
\end{equation}

We begin by establishing some a priori facts concerning the heat
kernel of a rough metric.

\subsection{Heat kernels for the rough metric Laplacian}

The Laplacian for a rough metric
is the non-negative self-adjoint operator $\Lap_\mg = -\divv_\mg \conn$,
the operator associated with the real-symmetric form
$J[u,v] = \inprod{\conn u, \conn v}_\mg$. 
To be precise about the results we establish here, 
we recall the defining features of heat kernels. 
We say that $\hk_\mg: \R_+ \times \cM \times \cM$ 
is the heat kernel of $\Lap_\mg$ 
if it is the minimal 
solution $\hk^\mg_t: \cM \times \cM \to \R_{\geq 0}$
to the heat equation 
\begin{equation}
\tag{HK}
\label{Def:HK}
\begin{aligned}
&\partial_t \hk^\mg_t(x,) = \Lap_\mg \hk^\mg_t(x,\mdot) \\
&\lim_{t \to 0} \hk^\mg_t(x,\mdot) = \ddelta_x,
\end{aligned}
\end{equation}
 where $\ddelta_x$ is the Dirac mass at $x \in \cM$,  satisfying 
\begin{equation*}
\hk^\mg_t(x,y) = \hk^\mg_t(y,x),\quad
\hk^\mg_t(x,y) \geq 0,\quad
\text{and}\quad \int_{\cM} \hk^\mg_t(x,y)\ d\mu_\mg(y) = 1.
\end{equation*}
Given an initial $u_0 \in \Lp{2}(\cM)$,
we are able to write
$$
\e^{-t\Lap_\mg}u_0(x) = \int_{\cM} \hk^\mg_t(x,y) u(y)\ d\mu_\mg(y)$$
for almost-every $x \in \cM$.

The following guarantees the regularity properties of the heat kernel 
when it exists.

\begin{theorem}
\label{Thm:HKExist}
If the heat kernel for $\Lap_\mg$ exists, 
then for each $t > 0$, there exists an $\alpha > 0$ such that 
$\hk^\mg_t \in \Ck{\alpha}(\cM^2)$.
\end{theorem}

This result is a direct consequence of
Theorem 5.3 in \cite{SC} by Saloff-Coste, by 
rewriting the Laplacian $\Lap_\mg$ as a divergence
form operator with bounded measurable coefficients
against a smooth background.

Before we present
the main theorem regarding
the existence of the metric $\mg_t$ when the initial 
metric $\mg$ is rough, we present the following
two lemmas which allow us to 
assert the non-degeneracy of $\mg_t$, when it exists.

\begin{lemma}[Backward uniqueness of the heat flow]
\label{Lem:HF1}
Let $u_t \in \Lp{2}(\cM)$ be a strict solution to the
heat equation
$ \partial_t u_t = \Lap_{\mg} u_t$
with $\lim_{t \to 0} u(t,x) = \Xi$,
where $\Xi \in \Sob{-1,2}(\cM)$ is 
a distribution. 
If there exists some $t_0 > 0$ such that
$u_{t_0} = 0$, then 
$u_t = 0$ for all $t > 0$ 
and $\lim_{t \to 0} u(t,\mdot) = 0$
in the sense of distributions.
\end{lemma}
\begin{proof}
First, suppose that $\Xi = v \in \Lp{2}(\cM)$.
Then, $u(t,x) = \e^{-t\Lap_{\mg}}v$
and we note that
$$\inprod{v, \e^{-t\Lap_{\mg}}v}_\mg
	= \inprod{v, \e^{-\frac{1}{2}t\Lap_{\mg}} \e^{-\frac{1}{2}t\Lap_{\mg}} v}_\mg
	= \inprod{\e^{-\frac{1}{2}t\Lap_{\mg}}v, \e^{-\frac{1}{2}t\Lap_{\mg}} v}_\mg
	= \norm{\e^{-\frac{1}{2}t\Lap_{\mg}}v},$$
where the second equality follows by the self-adjointness of $\Lap_\mg$.
Thus, at $t = t_0$, we obtain that
$\norm{\e^{-\frac{1}{2}t_0 \Lap_{\mg}}v} = 0$
and by induction,
$\e^{\frac{1}{2^n}t_0 \Lap_{\mg}v} = 0$.
Hence,
$$
v = \lim_{t \to 0} \e^{-t\Lap_{\mg}}v 
	= \lim_{n \to \infty} \e^{\frac{1}{2^n}t_0 \Lap_{\mg}v} = 0.$$

Now, for the case of an arbitrary distribution $\Xi$,
we note that for $s > 0$, 
$u_{t+s} = \e^{-t\Lap_{\mg}}u_s$ and therefore,
applying our previous argument with $v = u_s$, 
we obtain that $u_s = 0$ for every $s > 0$.
Now, fix $f \in \Ck[c]{\infty}(\cM)$, a test function,
and since $\inprod{\mdot,\mdot}$ extends continuously
to a pairing $\Sob{-1,2}(\cM) \times \Ck[c]{\infty}(\cM)$,
we obtain that 
$$
0 = \lim_{t \to 0} \int_{\cM} u_t(x) f(x)\ d\mu_\mg(x).$$
That is, $\Xi = 0$.
\end{proof}
 
Also, we have the following.

\begin{lemma}
\label{Lem:HF2}
The function $y \mapsto \hk^\mg_t(x,y) \in \dom(\Lap_\mg)$
and for all $t > 0$, $\partial_t \hk^\mg_t(x,\mdot) = \Lap_{\mg} \hk^\mg_t(x,\mdot)$
for each $x \in \cM$. 
If $\emptyset \neq \cN$ is an open subset 
on which $(x,y) \mapsto \hk^\mg_t \in \Ck{k}(\cN^2)$ (for $k \geq 1$), 
then for every $x \in \cN$ and $v \in \tanb_x \cM$, 
$y \mapsto (\extd_x\hk^\mg_t(x,y))(v)$ solves
$$ \partial_t (\extd_x\hk^\mg_t(x,\mdot))(v) 
	= \Lap_{\mg} (\extd_x\hk^\mg_t(x,\mdot))(v),
\qquad \lim_{t \to 0} (\extd_x\hk^\mg_t(x,\mdot))(v) = D_{x,v},
$$
where $D_{x,v} \in \Sob{-1,2}(\cM)$ is given by $D_{x,v}f = (\extd_x f)(v)$.
\end{lemma}
\begin{proof}
The fact that $x \mapsto \hk^\mg_t(x,y) \in \dom(\Lap_\mg)$
and $\partial_t \hk^\mg_t(x,\mdot) = \Lap_{\mg} \hk^\mg_t(x,\mdot)$
for $t > 0$ is by definition that $\hk^\mg_t: \cM^2 \to \R_+$
is the fundamental solution to the heat equation.

First, note that
$$
\partial_t (\extd_x\hk^\mg_t(x,\mdot))(v)
	= \extd_x (\partial_t \hk^\mg_t(x,\mdot))(v)
	= \extd_x (\Lap_\mg \hk^\mg_t(x,\mdot))(v).$$
Now, fix $u \in \dom(\Lap_\mg)$ and, 
fix a curve $\gamma:(-\epsilon,\epsilon) \to \cN$
such that $\gamma(0) = x$, $\dot{\gamma}(0) = v$, 
and observe that 
\begin{align*}
\inprod{(\extd_x\Lap_\mg\hk^\mg_t(x,\mdot))(v),u}_\mg
	&= \int_{\cM} (\extd_x\Lap_\mg\hk^\mg_t(x,y))(v)u(y)\ d\mu_\mg(y) \\
	&= \int_{\cM} \ddt[s]\rest{s=0} \Lap_\mg\hk^\mg_t(\gamma(s),y) u(y)\ d\mu_\mg(y) \\
	&= \ddt[s]\rest{s=0} \int_{\cM} \Lap_\mg\hk^\mg_t(\gamma(s),y) u(y)\ d\mu_\mg(y) \\
	&= \ddt[s]\rest{s=0} \int_{\cM} \hk^\mg_t(\gamma(s),y) \Lap_\mg u(y)\ d\mu_\mg(y) \\
	&= \int_{\cM} \ddt[s]\rest{s=0}  \hk^\mg_t(\gamma(s),y) \Lap_\mg u(y)\ d\mu_\mg(y) \\
	&= \inprod{(\extd_x\hk^\mg_t(x,\mdot))(v), \Lap_\mg u}_\mg.
\end{align*}
This shows that $u \mapsto \inprod{(\extd_x\hk^\mg_t(x,\mdot))(v), \Lap_\mg u} = 
\inprod{(\extd_x\Lap_\mg\hk^\mg_t(x,\mdot))(v),u}_\mg$
is continuous in $u$ and hence $(\extd_x\hk^\mg_t(x,\mdot))(v) \in \dom(\Lap_\mg)$
and by a similar calculation,
$$ \inprod{\partial_t (\extd_x\Lap_\mg\hk^\mg_t(x,\mdot))(v), u}_\mg 
	= \inprod{\Lap_\mg (\extd_x\hk^\mg_t(x,\mdot))(v),u}_\mg.$$
Since $\dom(\Lap_\mg)$ is dense in $\Lp{2}(\cM)$, 
we obtain that $(\extd_x\Lap_\mg\hk^\mg_t(x,\mdot))(v)$ solves
the heat equation.

Now, fix $f \in \Ck[c]{\infty}(\cM)$. Then,
\begin{align*}
\lim_{t \to 0} \int_{\cM} (\extd_x\hk^\mg_t(x,y))(v) f(y)\ d\mu_\mg(y)
	&= \lim_{t \to 0} \int_{\cM} \ddt[s]\rest{s =0} \hk^\mg_t(\gamma(s),y)f(y)\ d\mu_\mg(y) \\
	&= \lim_{t \to 0} \ddt[s]\rest{s = 0} \int_{\cM} \hk^\mg_t(\gamma(s),y)f(y)\ d\mu_\mg(y) \\
	&= \ddt[s]\rest{s = 0} \lim_{t \to 0}  \int_{\cM} \hk^\mg_t(\gamma(s),y)f(y)\ d\mu_\mg(y) \\
	&= \ddt[s]\rest{s = 0} f(\gamma(s)) \\
	&= (\extd_x f)(v).
\end{align*}
\end{proof}

In order to apply the elliptic tools we've described in the
previous sections, we need to assert 
that $\hk^\mg_t(x,y) > 0$. This
is the content of the following lemma. 

\begin{lemma}
\label{Lem:HKBd}
For each $t > 0$, there exist $0 < \kappa_t, \Lambda_t < \infty$ such that
$\kappa_t \leq \hk^\mg_t(x,y) \leq \Lambda_t$.
\end{lemma}
\begin{proof}
Arguing as in Theorem 5.2.1 in \cite{Davies} by Davies, 
we obtain that $\hk^\mg_t(x,y) > 0$ for all $x, y \in \cM$.
By the compactness of $\cM$ and since $\hk^\mg_t$ is at least continuous,
 we obtain that 
\begin{align*}
&\kappa_1 = \inf_{x,y \in \cM} \hk^\mg_t(x,y) = \min_{x,y\in \cM} \hk^\mg_t(x,y) > 0,\ \text{and}\\
&\Lambda_1 = \sup_{x,y \in \cM} \hk^\mg_t(x,y) = \max_{x,y \in \cM} \hk^\mg_t(x,y) < \infty.
\end{align*}
This completes the proof. 
\end{proof}

\subsection{The flow}
\label{Sec:GMFlow}

We collate the results we have obtained
so far and present the following 
existence and regularity theorem for $\mg_t$.
Recall that in the hypothesis of Theorem \ref{Thm:GMFlow}, 
we assume that $(x,y) \mapsto \hk^\mg_t(x,y) \in \Ck{0,1}(\cM^2)$
and that $(x,y) \mapsto \hk^\mg_t(x,y) \in \Ck{k}(\cN^2)$. 

\begin{proof}[Proof of Theorem \ref{Thm:GMFlow}]
First, we show that for each $t > 0$, $x \in \cM$ and $v \in \tanb_x\cM$,
there exists a unique $\phi_{t,x,v} \in \Sob{1,2}(\cM)$
which solves \eqref{Def:E}.
By Lemma \ref{Lem:HKBd}, we are able
to apply Proposition \ref{Prop:EUF}
on setting $\omega_x(y) = \omega(x,y) = \hk^\mg_t(x,y)$
and $\eta(y) = \extd_x(\hk^\mg_t(x,y))(v)$.

The only thing that needs to be to be checked is
that $\int_{\cM} \eta\ d\mu_\mg = 0$. 
In order to do so, let $\gamma: I \to \cM$ be a curve
so that $\gamma(0) = x$ and $\dot{\gamma}(0) = v$.
Then, on noting that 
$$(\extd_x(\hk^\mg_t(x,y)))(v) = \ddt[s]\rest{s=0} \hk^\mg_t(\gamma(s),y),$$
we compute 
\begin{multline*}
\int_{\cM} \ddt[s]\rest{s=0} \hk^\mg_t(\gamma(s),y)\ d\mu_\mg
	= \int_{\cM} \ddt[s]\rest{s=0} \hk^\mg_t(\gamma(s),y)\ d\mu_\mg \\
	= \ddt[s]\rest{s=0} {\int_\cM \hk^\mg_t(\gamma(s),y)\ d\mu_\mg}
	= \ddt[s]\rest{s= 0} 1 = 0,
\end{multline*}
where in the second line,
we have used the dominated convergence theorem to interchange
the integral and the limit involved in differentiation.
Thus, on invoking Proposition \ref{Prop:EUF}, we obtain 
a unique solution $\phi_{t,x,v} \in \Sob{1,2}(\cM)$
with $\int_{\cM} \phi_{t,x,v}\ d\mu_{\mg} = 0$.
It is easy to see then that $\mg_t$ is symmetric
at each $x$. 

Next, we show that $\conn \phi_{t,x,v} = 0$
if and only if $v = 0$.
Fix $v \neq 0$, and
recall Lemmas \ref{Lem:HF1} and \ref{Lem:HF2}
to conclude that $(\extd_x\hk^\mg_t(x,\mdot))(v) \neq 0$
since $\lim_{t \to 0} (\extd_x\hk^\mg_t(x,\mdot))(v) = D_{x,v} \neq 0$.
Since the solution provided by Proposition
\ref{Prop:EUDiv} is obtained by
inverting the one-one operator $\Div_A^R$ 
in Proposition \ref{Prop:EUDiv},
we must have that $\psi_{t,x,v} \not\in \nul(\Div_A^R) = \nul(\conn)$.
It is easy to see that if $v = 0$, then 
$(\extd_x\hk^\mg_t(x,\mdot))(v) = 0$ and hence, 
$\phi_{t,x,v} = 0$.
This shows that $\mg_t(u,u)(x) \geq 0$ and 
$\mg_t(u,u)(x) = 0$ if and only if $u = 0$.
That is, $\mg_t$ is non-degenerate.

Furthermore, it is easy to see that,
for $\alpha \neq 0$, 
$\alpha \phi_{t,x,v}$ solves
the equation \eqref{Def:E} with 
source term $(\extd_x\hk^\mg_t(x,\mdot))(\alpha v)$
and that, by linearity of
the equation \eqref{Def:E},
$\phi_{t,x,v} + \phi_{t,x,u}$
solves \eqref{Def:E} with
source term $(\extd_x\hk^\mg_t(x,\mdot))(u + v)$.
Hence, 
$\alpha \phi_{t,x,v} = \phi_{t,x,\alpha v}$
and $\phi_{t,x,v} + \phi_{t,x,u} = \phi_{t,x,u+v}$.
That is, 
$\mg_t(\alpha u,v)(x) = \alpha \mg_t(u,v)(x)$
and $\mg_t(u+ v, w)(x) = \mg_t(u,w)(x) + \mg_t(v,w)(x)$.
Thus, $\mg_t$ is linear in the first variable.

By combining symmetry, non-degeneracy, and linearity
in the first variable shows that $\mg_t(x): \tanb_x\cM \times \tanb_x \cM \to \R_{>0}$
defines an inner product on $\tanb_x\cM$ and
hence, a Riemannian metric.

Regularity is then a simple consequence of Theorem \ref{Thm:Reg}
since 
$\modulus{u}_{\mg_t(x)}^2 = \inprod{\eta_{x,u}, \phi_{x,u}}$, 
and the same regularity can be obtained for $x \mapsto \mg_t(u,v)(x)$
via polarisation.
\end{proof}

\section{Regularity of the flow for sufficiently smooth metrics}
\label{Sec:HKReg}

 In \cite{GM}, the authors demonstrate that the flow 
$\mg_t$ is smooth  for all positive times when starting with a smooth initial metric. 
We demonstrate  a similar result but when the initial metric is
assumed  to be $\Ck{k,\alpha}$, where $k \geq 1$.
Our approach is to demonstrate that we are able to localise
our weak solutions and then apply Schauder theory to 
obtain higher ($k+1$) regularity for the heat kernel $\hk_t^\mg$.
On applying  Theorem \ref{Thm:Reg}, 
we are able to assert that $\mg_t$ remains $\Ck{k}$.

\subsection{Higher regularity of the flow
for $\Ck{1}$ heat kernels}
\label{Sec:HighReg}

First, we demonstrate that for a  heat kernel that is $\Ck{1}$ everywhere, 
the regularity theorem (see Theorem \ref{Thm:Reg}) improves
from $\Ck{k-1,1}$ to $\Ck{k}$.

Recall that
$\Div_x u = \divv_\mgt (B \theta \omega_x)\conn u$.
We estimate the difference between such operators.
We fix $f: \cM \times \cM \to \R$ differentiable
with $-\Lambda \leq f(x,y) \leq \Lambda$
 where $\Lambda > 0$ and for $x,y \in U$, where $U$ is an open set.
Define $\Xi:U \times U \to \R_{\geq 0}$
by
$$\Xi(x,y) = \norm{f_x - f_y}_\infty + \norm{\conn(f_x - f_y)}_\infty,$$
where $f_x = f(x,\mdot)$.

\begin{lemma}
\label{Lem:OpDiff}
Let $(x,y) \mapsto f_x(y) \in \Ck{1}(\cM^2)$
and let $\Lambda > 0$ such that
$ -\Lambda \leq f_x(y) \leq \Lambda$
for $x \in U$, where $U$ is an open set
and all $y \in \cM$.
Define
$T_x u = -\divv_\mg f_x \conn u$ with domain
$\dom(T_x) = \dom(\Lap_\mg)$. Then, 
whenever $u \in \dom(\Lap_\mg)$, 
$$ \norm{T_x u - T_y u} \lesssim \Xi(x,y) \norm{u}_{\Lap_\mg}.$$
whenever $x, y \in U$ and where the
implicit constant depends on $U$.
\end{lemma}
\begin{proof}
Define $F_x = f_x + 2\Lambda$ and it 
follows that $\Lambda \leq F_x \leq 3\Lambda$
on $U \times \cM$.
On setting $S_x u = -\divv_\mg F_x \conn u$, by Proposition \ref{Prop:ConstDom}, 
we obtain that $\dom(S_x) = \dom(\Lap_\mg)$
and $S_x u = F_x \Lap_\mg u + \mg(\conn u, \conn F_x)$.
It is easy to check that 
$T_x u = S_x u - 2\Lambda \Lap_\mg u$ and therefore, 
$T_x u - T_y u = S_x u - S_y u$. 

We compute, 
$$
\norm{S_x u - S_y u} 
	\leq \norm{(F_x - F_y) \Lap_\mg u} + \norm{\mg(\conn u, \conn (F_x - F_y)},$$
but $F_x - F_y = f_x - f_y$, and hence, it follows that
$$\norm{T_x u - T_y u} 
	\leq \norm{f_x - f_y}_{\infty} \norm{\Lap_\mg u}  + \norm{\conn(f_x - f_y)}_\infty \norm{\conn u}.$$
The estimate 
$ \norm{\conn u} \lesssim \norm{u}_{\Lap_\mg}$ is trivial, and so 
the claim is proved. 
\end{proof}

We have a similar result for the resolvents $\Div_x^{-1}$
on the range of the operator $\Div_x$. 

\begin{lemma}
\label{Lem:OpResDiff}
Suppose that $(x,y) \mapsto \omega_x(y) \in \Ck{1}(\cM^2)$ and 
let $u_1, u_2 \in \Lp{2}(\cM)$ satisfy $\int_{\cM} u_1\ d\mu_\mgt = \int_{\cM} u_2\ d\mu_\mgt = 0$.
Then, 
$$\norm{\Div_x^{-1} u_1 - \Div_y^{-1}u_2} \lesssim \Xi(x,y) \norm{u_1} + \norm{u_1 - u_2}.$$
The implicit constant is independent of $x$ 
and this expression is valid for all $x, y \in \cM$.
\end{lemma}
\begin{proof}
First note that for $v \in \dom(\Lap_\mg) = \dom(\Div_x)$, 
$$\norm{\Div_x v - \Div_y v} = \norm{ \theta (\Dir_x v - \Dir_y v)}
	\lesssim \norm{ \Dir_x v - \Dir_y v}
	\lesssim \Xi(x,y) \norm{v}_{\Lap_\mg}$$
by invoking Lemma \ref{Lem:OpDiff} with $U = \cM$.

Now, fix $u \in \Lp{2}(\cM)$ with $\int_{\cM} u\ d\mu_\mgt = 0$
and note that $\Div_x^{-1}u = \Div_x^{-1}(\Div_y \Div_y^{-1})u =  (\Div_x^{-1}\Div_y) \Div_y^{-1})u$.
Also, $\Div_y^{-1} u = \Div_x^{-1} \Div_x L^{-1}_y u$
since the resolvent and operator commute on its domain.
Thus,
\begin{multline*}
\norm{\Div_x^{-1} u - \Div_y^{-1} u}
	= \norm{\Div_x^{-1} \Div_y \Div_y^{-1} u - \Div_x^{-1} \Div_x \Div_y^{-1}u} 
	= \norm{\Div_x^{-1}(\Div_y - \Div_x)\Div_y^{-1} u} \\
	\lesssim \norm{(\Div_y - \Div_x)\Div_y^{-1}u}
	\lesssim \Xi(x,y) \norm{\Div_y^{-1} u}_{\Lap_\mg}, 
\end{multline*} 
since by Lemma \ref{Lem:Uni}, we have that
$\norm{\Div_x^{-1} u} \lesssim \norm{u}$
independent of $x$.
By Proposition \ref{Prop:UniBdd}, we obtain that
$\norm{v}_{\Lap_\mg} \simeq \norm{v}_{\Dir_y}$,
independent of $y$ and that $\norm{v}_{\Dir_y} \simeq \norm{v}_{\Div_y}$.
Hence, on setting $v = \Div^{-1}_y u$, we obtain that
$\norm{\Div_y^{-1} u}_{\Lap_\mg} \lesssim \norm{u}$.

Now, for $u_1$ and $u_2$ as in the hypothesis,
\begin{multline*}
\norm{\Div_x^{-1} u_1 - \Div_y^{-1} u_2}
	\leq \norm{\Div_x^{-1} u_1 - \Div_y^{-1} u_1} + \norm{\Div_y^{-1} (u_1 - u_2)} \\
	\lesssim \Xi(x,y) \norm{u_1} + \norm{u_1 - u_2}.
\end{multline*}
\end{proof}

With the aid of these two lemmas, we 
improve the regularity
from Theorem \ref{Thm:Reg} as follows.

\begin{theorem}
\label{Thm:BetReg}
Suppose that $(x,y) \mapsto \omega_x(y) \in \Ck{1}(\cM^2)$
and that $x\mapsto \omega_x \in \Ck{k}(\cN)$
for $k \geq 1$.
Moreover, suppose that $(x,y) \mapsto \eta_x(y) \in \Ck{0}(\cN \times \cM)$
and that $x \mapsto \eta_x \in \Ck{l}(\cN)$
for $l \geq 0$.
If at $x \in \cN$,  $\phi_x$ solves \eqref{Def:F}
with $\int_{\cM} \phi_x\ d\mu_\mg = \int_{\cM} \eta\ d\mu_\mg = 0$.
Then, $x \mapsto \inprod{\eta_x, \phi_x} \in \Ck{\min\set{k,l}}(\cM)$.
\end{theorem}
\begin{proof}
First, suppose that $l = 0$. Then, we show
that $x \mapsto \inprod{\eta_x, \phi_x} \in \Ck{0}(\cN)$.
For that, note that 
\begin{multline*}
\modulus{\inprod{\eta_x, \phi_x} - \inprod{\eta_y, \phi_y}}
	 \leq \modulus{\inprod{ \eta_x - \eta_y, \phi_x}} + 
		\modulus{\inprod{\eta_y, \phi_x - \phi_y}} \\
	\leq \norm{\eta_x - \eta_y} \norm{\phi_x} 
		+ \norm{\eta_y}\norm{\phi_x - \phi_y}.
\end{multline*}
Since we assume that $(x,z) \mapsto \eta_x(z)\in \Ck{0}(\cN \times \cM)$,
the same is true for $(x,z) \in \close{U} \times \cM$,
where $\close{U} \subset \cN$ is a compact set with nonempty interior
containing $x$ and hence, $\inprod{\eta_x - \eta_y}$
can be made small.
Now, the term $\phi_x = \Div_x^{-1} \theta \eta_x - \fint_{\cM} \Div_x^{-1} \theta\eta_x\ d\mu_\mg$
and hence, it suffices to show that
$\norm{\Div_{x}^{-1} \theta \eta_x - \Div_{y}^{-1} \theta \eta_y}$
can be made small.
For this, note that 
$$
\norm{\Div_{x}^{-1} \theta \eta_x - \Div_y^{-1} \theta\eta_y}
	\lesssim \Xi(x,y) \norm{\eta_x} + \norm{\eta_x - \eta_y}$$
by Lemma \ref{Lem:OpResDiff}, and hence,
this term can also be made small when $y$ is sufficiently 
close to $x$.

Next, we note that by bootstrapping,
it suffices to consider the  situation where $k, l = 1$
and we note that Theorem \ref{Thm:Reg} gives us that
$x \mapsto \inprod{\eta_x, \phi_x}$ has a 
bounded derivative. All we need to prove is 
that this derivative is continuous.

We recall that, via the product rule for the weak
derivative, we write inside a chart,
$$
(\partial_i \inprod{\eta_x, \phi_x})(v) =
	\inprod{(\partial_i \eta_x, \phi_x} + \inprod{\eta_x, \partial_i \phi_x},$$
and hence, we show that each term of the right hand side is
continuous.

Fix $x,y \in U \subset \cM$ open , where $(\psi,\tilde{U})$ is
a chart with $\close{U} \subset \tilde{U}$ compact.
Then, we have
$$\inprod{ \partial_i \eta_x, \phi_x} - \inprod{\partial_i \eta_y, \phi_y}
	\leq \inprod{ \partial_i\eta_x - \partial_i \eta_y, \phi_x} + 
		\inprod{\partial_i \eta_y, \phi_x - \phi_y}.$$
Now, since $x \mapsto \eta_x$ is $\Ck{1}$ by assumption, 
$\norm{\partial_i \eta_x - \partial_i\eta_y}$ can be made small.

In the continuity case, we have already 
shown that $\norm{\phi_x - \phi_y}$ can be
made small, so we consider
the next term
$$
\inprod{ \eta_x, \partial_i \phi_x} - \inprod{\eta_y, \partial_i \phi_y}
	= \inprod{\eta_x - \eta_y, \partial_i \phi_x} + \inprod{\eta_y, \partial_i \phi_x - \partial_i \phi_y}.$$
Now, it is easy to see that the first term on the right hand side
is trivially continuous because $\norm{\eta_x - \eta_y}$ can be made small. 
The continuity for the second term follows 
by showing that 
$\norm{\partial_i\phi_x - \partial_i \phi_y}$ can be made small. 
Recall that $\partial_i \phi_x = \Div_{x}^{-1}\theta \eta_{x,i}' - \fint_{\cM} \Div_x^{-1}\theta \eta_{x,i}'$,
where $\eta_{x,i}'  = \partial_i \eta_x - (\partial_i \Dir_x)\phi_x$.
Thus, it suffices to prove that $\norm{\Div_x^{-1}\theta\eta_{x,i}'- \Div_{y}^{-1}\theta\eta_{y,i}'}$
can be made small. By Lemma \ref{Lem:OpResDiff},
we have that
$$\norm{\Div_x^{-1}\theta\eta_{x,i}- \Div_{y}^{-1}\theta\eta_{y,i}}
	\lesssim \Xi(x,y) \norm{\eta_{x,i}} 
	+ \norm{\Div_y^{-1} \theta(\eta_{x,i} - \eta_{y,i})}.$$
Hence, we are reduced to proving
that $\norm{\eta_{x,i}' - \eta_{y,i}'}$
can be made small.

Now, note that  
$(\partial_i \Dir_x)\phi_x  = (\partial_i \Dir_x)\Div_x^{-1}\theta\eta_x$
since  $(\partial_i \Dir_x) (\fint_{\cM}\Div_x^{-1}\theta\eta_x\ d\mu_\mg) = 0$
and thus, 
$$\norm{\eta_{x,i}' - \eta_{y,i}'}
	\leq \norm{\partial_i \eta_x - \partial_i \eta_y} + 
	\norm{(\partial_i \Dir_x)\Div_x^{-1}\theta\eta_x + (\partial_i \Dir_y)\Div_y^{-1}\theta\eta_y}.$$
It is easy to see that the first term can be made small, so we
only need to show that the second term can be made small. 
Now,
\begin{multline*}
\norm{(\partial_i \Dir_x)\Div_x^{-1}\theta\eta_x - (\partial_i \Dir_y)\Div_y^{-1}\theta\eta_y} \\
	\leq 
	\norm{[(\partial_i \Dir_x) - (\partial_i \Dir_y)]\Div_x^{-1}\theta\eta_x } +
		\norm{(\partial_i \Dir_y)(\Div_x^{-1} \theta \eta_x -\Div_y^{-1}\theta\eta_y)},
\end{multline*}
and by Lemma \ref{Lem:OpDiff}, we have that
$$ \norm{[(\partial_i \Dir_x) - (\partial_i \Dir_y)]\Div_x^{-1}\theta\eta_x }
	\lesssim \Xi(x,y)\norm{\Div_x^{-1} \theta \eta_x}_{\Lap_\mg}
	\lesssim \Xi(x,y) \norm{\eta_x}.$$
For the remaining term, as in the proof of 
Lemma \ref{Lem:OpResDiff}, we write,
$$
\Div_x^{-1} \theta \eta_x -\Div_y^{-1}\theta\eta_y
	= \Div_{x}^{-1}(\Div_x - \Div_y)\Div_y^{-1}\theta\eta_x + \Div_y^{-1}\theta(\eta_x - \eta_y),$$
and since $\norm{(\partial_i \Dir_y)\Div_x^{-1}} \lesssim 1$ uniformly in $x$
and $y$ inside $U$ and since $\close{U}$ is compact, we have that
\begin{multline*}
\norm{(\partial_i \Dir_y)(\Div_x^{-1} \eta_x -\Div_y^{-1}\theta\eta_y)}
	\lesssim \norm{(\Div_x - \Div_y)\Div_y^{-1}\theta\eta_x} + \norm{\Div_y^{-1}\theta(\eta_x - \eta_y)} \\
	\lesssim \Xi(x,y) \norm{\Div_y^{-1}\theta\eta_x}_{\Lap_\mg} + \norm{\eta_x - \eta_y} 
	\lesssim \Xi(x,y) \norm{\eta_x} + \norm{\eta_x - \eta_y}.
\end{multline*}
This is again a quantity that can be made small.
This shows that $x \mapsto \eta_{x,i}$ is continuous
and to show that the $(\min\set{k, l})$-th derivative
can be made continuous for $k, l \geq 1$
is obtained via a bootstrapping of this procedure.
\end{proof}

\begin{remark}
Showing higher derivatives are continuous
is a rather tedious task. 
One considers the expression solving for a 
second derivative (when there is sufficient
regularity in $x\mapsto \omega_x$
and $x \mapsto \eta_x$) given by 
$$ \Dir_x \partial_j \partial_i \phi_x = \partial_j \partial_i \eta_x
	- (\partial_j \partial_i \Dir_x) \phi_x 
	- (\partial_i \Dir_x) \partial_j \phi_x
	- (\partial_j \Dir_x) \partial_i \phi_x.$$
The first term on the right hand side
can be handled easily. The second
term follows from a similar estimate
as in Theorem \ref{Thm:BetReg}, because
$(\partial_j \partial_i \Dir_x)$
is a divergence form operator
$-\divv_{\mg} (\partial_j \partial_i \omega_x) \conn u$,
whose coefficients satisfy 
$-\Lambda_{U} \leq \omega_x(y) \leq \Lambda_{U}$
for $x,y \in U$, an open neighbourhood of $x$ 
for which $\close{U} \subset \cN$.
The remaining two terms can also 
be handled similarly on writing $\partial_j \phi_x$
and $\partial_i \phi_x$ as a solution
via the resolvent terms $\Div_{x}^{-1}$
to relate back to $\eta_{x,i}'$
and to $\eta_{x}$.
\end{remark}

As a corollary, we obtain an improvement of
the regularity of the flow for $\Ck{1}$ heat
kernels. This is the statement 
of Theorem \ref{Thm:GMFlowReg},  which is readily checked
to be a direct  consequence of Theorem \ref{Thm:BetReg}.

\subsection{Heat kernel regularity in terms of the regularity of the metric}
\label{Sec:HKRegMet}

In this subsection, we relate the regularity of the
heat kernel to the regularity of the metric.
We first prove the following important localisation lemma. 

\begin{lemma}
\label{Lem:Loc}
Suppose that $\divv_\mg A \conn u = f$, for $u \in \Sob{1,2}(\cM)$ and
$f \in \Lp{2}(\cM)$. Then, for each $x \in \cM$,
there is an $r > 0$ and a chart $\psi: U \to B_r(x')$
where $x' = \psi(x)$ and 
such that on $\Omega = \psi^{-1}(B_{1/2r}(x'))$,
$$
\divv_{\mgt,\Omega} A B\theta \conn u = \theta f$$
in $\Lp{2}(\Omega,\mgt)$,
where $\mgt = \pullb{\psi} \delta$, the pullback of
the Euclidean metric in $B_r(x')$, 
$d\mu_\mg = \theta\ d\mgt$ and
$\mgt(B u, v) = \mg(u,v)$.
Moreover, this equation holds if and only if
$$\divv_{\R^n,B_{1/2r}(x')} \tilde{A}\tilde{B} \tilde{\theta} \conn \tilde{u} = \tilde{\theta} \tilde{f},$$ 
where $\tilde{\xi} = \eta (\xi \comp \psi^{-1})$, where $\eta$ is a smooth 
cutoff which is $1$ on $B_{1/2r}(x)$, and $0$ outside $B_{3/4r}(x)$.
\end{lemma}
\begin{proof}
Fix $v \in \Ck[c]{\infty}(\Omega)$. Then, for $u \in \dom(\divv_\mg)$,
we have that 
$$ \inprod{\divv_\mg u, v} = \inprod{u, \conn v} = \inprod{u, \conn v}_{\Lp{2}(\Omega,\mg)}
	= \inprod{B\theta u,  \conn v}_{\Lp{2}(\Omega, \mgt)}.$$
Since this holds for any such $v \in \Ck[c]{\infty}(\Omega)$,
it follows that $B \theta u \in \dom(\divv_{\mgt,\Omega})$
and hence
$$\inprod{B\theta u,  \conn v}_{\Lp{2}(\Omega, \mgt)} 
	= \inprod{\theta^{-1} \divv_{\mgt,\Omega} B\theta u, v}_{\Lp{2}(\Omega,\mg)}.$$
Thus, $\divv_\mg A \conn u = f$ implies that
$ \inprod{\divv_\mg A \conn u, v} = \inprod{f, v}$
for all $v \in \Ck[c]{\infty}(\Omega)$ and hence,
$ \divv_{\mgt,\Omega} A B\theta \conn u = \theta f$
in $\Lp{2}(\Omega,\mgt)$.
Since $\eta$ and $\phi$
induces a bijection between $\Ck[c]{\infty}(\B_{1/2r})$
and $\Ck[c]{\infty}(\Omega)$, it follows
that
$\divv_{\R^n,B_{1/2r}(x')} \tilde{A}\tilde{B} \tilde{\theta} \conn \tilde{u} = \tilde{\theta} \tilde{f}.$
\end{proof}

When the metric is sufficiently regular (i.e. at least Lipschitz), we are able to write
solutions to the Laplace equation in
non-divergence form.

\begin{lemma}
\label{Lem:NonDiv}
Let $(\psi, U)$ be a chart near $x$ with $\psi(U) = B_{1/2r}(x')$,
and suppose that $\mg \in \Ck{k,\alpha}(U)$, where $\alpha = 1$ if $k = 0$
and otherwise, for $k \geq 1$, $\alpha \in [0,1]$. 
Then, inside $\psi(U)$,
$$\widetilde{\Lap_\mg u}(y) 
	= \tilde{A}^{ij}(y) \partial_i \partial_j \tilde{u}(y) + \partial_j(\tilde{A}^{ij}\tilde{\theta}) \partial_i \tilde{u},$$
for almost-every $y  \in B_{1/2r}(x')$, where $\tilde{\xi}$ is the notation
from Lemma \ref{Lem:Loc}.
The coefficients $\tilde{A}^{ij}, \tilde{\theta}, \partial_j(\tilde{A}^{ij}) \in \Ck{k-1,\alpha}$
for $k \geq 1$. Otherwise,
$\tilde{A}^{ij}, \tilde{\theta}, \partial_j(\tilde{A}^{ij}) \in \Lp{\infty}(B_{1/2r}(x')).$
\end{lemma}
\begin{proof}
This is simply a direct consequence of Theorem 8.8 in \cite{GT}. 
This formula is precisely the one written in (8.18) in their
theorem.
\end{proof}

Next, we obtain the first increase in regularity which allows
us to initiate a bootstrapping procedure. 

\begin{lemma}
\label{Lem:Boot}
Let $(\psi,U)$ be a chart near $x$ and $\psi(U) = B_r$. 
Suppose that $\mg \in \Ck{k,\alpha}(U)$ for $k \geq 1$ and $\alpha \in [0,1]$
and suppose that $u \in \Sob{1,2}(\cM)$ and $f \in \Lp{2}(\cM)$
satisfy $\Lap_\mg u = f$.
If $\tilde{f} \in \Ck{\alpha}(B_{r}(x'))$, 
then $\tilde{u} \in \Ck{2,\alpha}(B_{1/4r}(x'))$.
Moreover, $u \rest{\Omega} \in \Ck{2,\alpha}(\Omega)$
where $\Omega = \psi^{-1}(B_{1/4r}(x'))$.
\end{lemma}
\begin{proof}
First, set $r' = 1/2r$, and invoke the localisation
from Lemma \ref{Lem:Loc}. Note that this equation
$\divv_{\R^n, B_{3/4r'}(x')} \tilde{AB\theta} \conn \tilde{u} = \tilde{\theta f}$
is in divergence form,
and since $\tilde{\theta f} \in \Ck{\alpha}(\close{B_{3/4r'}(x')})$,
we have that $\tilde{\theta f} \in \Lp{q}(B_{3/4r'}(x'))$
for any $q > n$. Hence,
we can invoke the elliptic Harnack estimate
from Theorem 8.22 in \cite{GT}
to obtain that $\tilde{u} \in \Ck{\beta}(B_{1/2r'}(x'))$
for some $\beta > 0$. 
 
Next, we invoke Lemma \ref{Lem:NonDiv},
to write 
$$\tilde{A}^{ij}(y) \partial_i \partial_j \tilde{u}(y) + \partial_j(\tilde{A}^{ij}\tilde{\theta}) \partial_i \tilde{u} 
	= \tilde{\theta} \tilde{f},$$
inside $B_{1/2r'}(x')$. Then, note that $\tilde{u}$
solves 
$$L \tilde{u} = \tilde{\theta} \tilde{f}, 
\quad\text{with}\quad\tilde{u} = \phi \in \Ck{0}(\bnd B_{3/8r'}(x')),$$
where $L$ has $\Ck{k-1,\alpha}$ coefficients and
$\tilde{\theta} \tilde{f} \in \Ck{\alpha}(B_{3/8r'}(x'))$
simply on setting $\phi = \tilde{u}$ on $\bnd B_{3/8r'}(x') \subset B_{1/2r'}(x')$
on which we have already proved that $\tilde{u}$ is $\Ck{\beta}$
and hence, continuous.

Thus, we can invoke Theorem 6.13 in \cite{GT}
to obtain that $\tilde{u} \in \Ck{2,\alpha}(B_{1/2r'}(x'))$.
By the definition of $\tilde{u}$, and on noting
that $r' = 1/2r$, 
the conclusions for $u\rest{\Omega}$ follow.
\end{proof}

With these tools in hand, we prove the following
main theorem of this section.
By $\beta'$ we denote
the a priori regularity of the heat kernel
obtained from Theorem \ref{Thm:HKExist}.

\begin{theorem}
\label{Thm:HKReg}
Let $\mg \in \Ck{k,\alpha}(\cN)$, 
where $\emptyset \neq \cN$ is an open set
and where $k \geq 1$ and $\alpha \in [0,1]$.
Then, $\hk^\mg_t \in \Ck{k+1,\beta}(\cN^2)$,where $\beta = \min\set{\alpha, \beta'}$. 
\end{theorem}
\begin{proof}
 By the regularity of $\mg$,  we know that the heat kernel exists
and that it is at least $\Ck{\beta'}$ for some
$\beta' > 0$ by Theorem \ref{Thm:HKExist}.

Fix $z \in \cM$ and set $ u(y) = \hk^\mg_t(y,z)$
and $f(y) = \partial_t \hk^\mg_t(y,z)$. 
Now, fix $(\psi,U)$, a chart near $x \in \cN$ so that 
$U \subset \cN$ and $B_{r}(x') = \psi(U)$

We proceed by applying Theorem 6.17 in \cite{GT}.
Define $\beta = \min\set{\beta',\alpha}$. 
First, let us apply the initial bootstrapping lemma, 
Lemma \ref{Lem:Boot} to conclude that,
in fact, $u\rest{\Omega} \in \Ck{2,\beta}(\Omega)$, 
where $\Omega = \psi^{-1}(B_{1/4r}(x'))$.
This shows that  that $u \in \Ck{2,\beta}(\cN)$
and by the symmetry of the heat kernel, 
we obtain that $\hk^\mg_t \in \Ck{2,\beta}(\cN^2)$.
Thus, we have shown that the 
conclusion holds for $k = 1$.

Now, in the case that $k = 2$, we have that the 
operator $L$ as defined in the proof of 
Lemma \ref{Lem:Boot} has $\Ck{1,\alpha}$-coefficients. Therefore, since
we have that $u\rest{\Omega}, f\rest\Omega \in \Ck{2,\beta}(\Omega)$
and, in particular, $u\rest{\Omega}, f\rest{\Omega} \in \Ck{1,\beta}(\Omega)$
by what we have just done, 
Theorem 6.17 in \cite{GT} yields that
$u\rest{\Omega} \in \Ck{3,\beta}(\Omega)$.

Now, to proceed by induction, suppose
we have that $u \in \Ck{k-1,\beta}$ and 
the metric $\mg \in \Ck{k,\alpha}$.
Then, $f \in \Ck{k-1,\beta}$ and 
the coefficients of $L$ are $\Ck{k-1, \alpha}$. Hence, Theorem 6.17 in \cite{GT}
gives that $u\rest{\Omega} \in \Ck{k+1,\beta}(\Omega)$.
That is, $\hk^\mg_t \in \Ck{k+1,\beta}(\cN^2)$.
\end{proof}

\section{$\RCD(K,N)$ spaces and singularities}
\label{Sec:RCD}

In this section, we first demonstrate
that the flow defined by \eqref{Def:GM}
is equal to the flow that
Gigli and Mantegazza define for
 $\RCD(K,N)$ spaces in \cite{GM}. 
In fact,  for a smooth initial 
metric, they verify this fact in their
paper.  We ensure
that this is true in our more general 
setting on an admissible region.

We then consider the flow 
defined as \eqref{Def:GM}
on manifolds with geometric singularities
away from the singular region.
The correspondence we establish between 
this and the flow of $\RCD(K,N)$ defined by Gigli-Mantegazza 
then allows us to assert that this 
flow can be described by an evolving
metric tensor away from the singular region
for certain $\mg_t$-admissible points.

\subsection{Correspondence to the flow for $\RCD(K,N)$ spaces}

First, we recall 
some terminology that will be
essential for the material 
we present here.
Let $(\Spa,\met,\mu)$ be a compact measure metric space, 
and denote set of probability measures 
by $\sP(\Spa)$. This set can be made into a metric space
under
$$ W_2(\nu,\sigma)^2 = \inf\set{\int_{\Spa \times \Spa} \met(x,y)^2\ d\pi: \pi\ \text{is a transport map from $\nu$ to $\sigma$}},$$
where by transport map, we mean that $\pi(A \times \Spa) = \nu(A)$
and $\pi(\Spa \times B) = \sigma(B)$.
The metric $W_2$ 
is the \emph{Wasserstein metric} 
and  the space $(\sP(\Spa),W_2)$ is the 
\emph{Wasserstein space}.
An important feature is that, when 
$\met$ is a length space, so is $(\sP(\Spa),W_2)$
and when $\met$ is a geodesic space,
then the same property holds for $(\sP(\Spa),W_2)$.

In their paper \cite{GM}, the authors
demonstrate that the flow defined by \eqref{Def:GM} 
for initial smooth metrics 
coincides with a flow which they define
as a heat-flow in Wasserstein space. 
Namely, they demonstrate that
$$ \mg_t(\gamma_s', \gamma_s') = \modulus{\dot{\nu}_s}^2,$$ 
where $\nu_s = \hk^\mg_t(\gamma_s,\mdot)\ d\mu$
and where $\modulus{\dot{\nu}_s}$ is the 
$W_2$ metric speed of the curve $\nu_s$.

In the following theorem, we
verify this is indeed the case 
when $(\cM,\mg)$ with $\mg$
rough and inducing a distance metric satisfying an $\RCD(K,N)$ condition.
The proof is essentially the same as in the proof of 
Theorem 3.6 in \cite{GM}, which 
in turn relies on the uniqueness of 
solutions of the continuity equation
stated as Theorem 2.5 in  \cite{GM}, 
when the underlying space is a Riemannian manifold
with a smooth metric. 
The proof of their Theorem 2.5
fails to  hold in our setting as they 
resort to Euclidean  results via the Nash embedding theorem
which we are unable  to do given the low regularity of our
metric. 

Moreover, we note that the set $\cN$
may not be convex with respect to $\mg_t$. 
Recall that two points $x, y \in \cM$
are $\mg_t$-admissible if for every 
absolutely continuous 
curve $\gamma: I \to \cM$
with $\gamma(0) = x$ and $\gamma(1) = y$,
there is another absolutely continuous curve 
$\gamma': I \to \cM$ 
with $\gamma'(s) \in \cN$ 
for $s$-a.e.
for which
$\len_{\met_t}(\gamma') \leq \len_{\met_t}(\gamma)$
where 
$$
\len_{\met_t}(\gamma) = \cbrac{\int_{\cM} \modulus{\dot{\gamma}(s)}_{\met_t}^2\ ds}^{\frac{1}{2}},$$
and where $\modulus{\dot{\gamma}(s)}_{\met_t}$ is the metric
speed of the curve computed
with respect to $\met_t$.
With this terminology at hand, we present the 
following important theorem.

\begin{theorem}
\label{Thm:RCD}
Let $(\cM,\mg)$ be a smooth manifold
with a rough metric and suppose that $\mg$
induces a length structure such that
$(\cM,\met_{\mg},\ d\mu_\mg)$ is 
$\RCD(K,N)$.
Let $\mg_t$ be the flow given by Theorem
\ref{Thm:GMFlow} on an open subset $\emptyset \neq \cN$. 
Suppose $s \mapsto \gamma_s \in \cM$ is an absolutely continuous curve
between two admissible points $x, y  \in \cM$ for which
$\gamma(s) \in \cN$ for $s$-a.e. 
Fix $t>0$ and define 
$$
	\nu_s := \hk^\mg_t(\gamma_s,\mdot)d\mu_\mg = H_t \left( \nu_{0,\gamma_s}  \right),
$$
where $H_t$ denotes the heat flow and $\nu_{0,\gamma_s} = \delta_{\gamma_s}$,
the delta measure at 
$\gamma_s$. Then, $s \mapsto \nu_s$ is absolutely continuous with respect to
$W_2$ and for $s$-a.e.,  
$$
		\mg_t(\dot{\gamma}_s , \dot{\gamma}_s)= \modulus{\dot{\nu}_s}^2.
$$ 
Moreover, 
$$\met_t(x,y)^2= \inf_{\gamma(s) \in \cN\ s-\text{a.e.}} \int_{\cM} \modulus{\dot{\gamma}_s}_{\mg_t}^2\ ds.$$
\end{theorem}

\begin{proof}
The absolute continuity of $\nu_s$ follows from absolute continuity of $\gamma_s$ and the 
contraction property of the heat flow in spaces with curvature bounded below. 
From Theorem \ref{Thm:GMFlow}, we know that there exist a family $\psi_{t,\gamma_s,\dot{\gamma}_s} \in \Sob{1,2}(\cM)$ 
solving the following equation (in the sense of distributions)
$$
	-\divv_{\mg} \hk^\mg_t(\gamma_s,\mdot) \conn \psi_{t,\gamma_s,\dot{\gamma}_s} = \extd_x(\hk^\mg_t(\gamma_s,\mdot))(\dot{\gamma}_s)
$$ 

Now, we note that $\nu_s$ has bounded compression i.e. $\nu_s \ll d\mu_\mg$ and 
since we assume that $(\cM, \met_\mg, d\mu_\mg)$ is an $\RCD(K,N)$ space,
the Sobolev space $\Sob{1,2}(\cM)$ is Hilbert. So 
applying Proposition 4.5 in \cite{Gigli-Han}, we have
$$
	\modulus{\dot{\nu_s}} = \norm{\nabla \psi_{t,\gamma_s,\dot{\gamma}_s}}_{\Lp{2}(\nu_s)},
$$
which in turn means that
$$
	|\dot{\nu}_s|^2 = \int_M\; |\nabla \psi_{t,\gamma_s, \dot{\gamma}_s}|^2 d \nu_s = \mg_t \left( \dot{\gamma}_s , \dot{\gamma}_s \right).
$$

As a direct consequence, we get
\begin{equation*}
\label{eq:compatibility}
	\met_{\mg_t}^2 \left( x,y \right) = \inf_\gamma \set{\int_0^1 |\dot{\nu}_s|^2\ ds:   \;  
		\gamma(s) \in \cN\ s\text{-a.e. joining $x$ and $y$}}
\end{equation*}

Notice that the right hand side the equation above is the definition of distance 
given by the flow \eqref{Def:GM}. So the proof is complete. 
\end{proof}

With this theorem at hand, 
and on collating results 
we have obtained previously, 
we give the following proof
of Theorem \ref{Thm:Main}.
 
\begin{proof}[Proof of Theorem \ref{Thm:Main}]
Since we assume that $(\cM,\met_\mg, d\mu_\mg)$ is an $\RCD(K,N)$ space,
we know from Theorem 7.3 in \cite{AGMR}
that $\hk^\mg_t \in \Ck{0,1}(\cM^2)$.
Moreover, we assume that $\mg \in \Ck{k}(\cM \setminus \cS)$
for $k \geq 1$,
and since $\cS \subsetneqq \cM$ is closed, 
$\cM \setminus \cS$ is open, and so we 
apply Theorem \ref{Thm:HKReg} to 
obtain that $\hk^\mg_t \in \Ck{k+1}(\cM^2)$.
By the assumptions we've made, $k + 1 \geq 2$
and hence, we invoke Theorem \ref{Thm:GMFlow} 
to obtain the conclusion.
Moreover, by Theorem \ref{Thm:RCD},
we are able to assert that $\met_t(x,y)$ is 
induced by $\mg_t$ for $\mg_t$-admissible
points $x, y \in \cM$.
\end{proof}

\subsection{Witch's hats and boxes}
\label{Sec:Examples}

In this section, we 
prove Corollary \ref{Cor:Witch} and \ref{Cor:Box}
from \S\ref{Sec:Sing}. 

First, we note the following theorem 
that will make our constructions easier.

\begin{proposition}\label{prop:petrunin-gluing}
The gluing of two Alexandrov spaces via an isometry between
their boundaries produces an Alexandrov space with the same lower curvature bound.
Moreover, such a space is an $\RCD(K,N)$ space.
\end{proposition}
\begin{proof}
The first part of the Proposition
concerning the gluing of Alexandrov
spaces is in \cite{Pet}
by  Petrunin. 
The curvature bounds of Lott-Sturm-Villani
follow from \cite{PetRCD} by 
the same author.
That an Alexandrov
space is $\RCD$ is due to 
\cite{KMS} by Kuwae, Machigashira, and Shioya.
\end{proof}

With this tool in hand, let us first consider the case of the box.
Let 
$$ B^{n} =\partial \left[-\sqrt{\frac{1}{2(n+1)}}, \sqrt{\frac{1}{2(n+1)}}~ \right]^{n+1}$$  
and $G: B^{n} \to \Sph^n \subset \R^{n+1}$ be the radial projection map defined by
$$
	G(x) = \frac{x}{\modulus{x}}.
$$
We have $B^{n} \subset \mathbb{B}^{n}_0(1) $ which means that $G$ is an expansion and hence
$$
	\met_{B^{n}} \left( x,y \right) \le \met_{\Sph^n} \left( G(x) , G(y) \right) \le \sqrt{2(n+1)}  \met_{B^{n}} \left( x,y \right). 
$$
The second inequality follows from the fact that  $\Sph^n \subset [-1,1]^{n+1} $.
Putting these together, we deduce that $G^{-1}: \Sph^n \to B^{n} $ is Lipschitz and 
 that the Lipschitz constant of $G^{-1}$ satisfies $\Lip(G^{-1}) \le 1 $.

Immediately, by Proposition \ref{prop:petrunin-gluing}, we obtain 
the proof of Corollary \ref{Cor:Box}.
\begin{proof}[Proof of Corollary \ref{Cor:Box}]
By Proposition \ref{prop:petrunin-gluing}, we obtain that
$B^{n}$ is an $\RCD(0,n)$ space.
Moreover, it is easy to see that the Riemannian
metric induced via $G$ coming from the sphere is smooth on $B$
away from the edges and corners.
Thus, we can apply Theorem \ref{Thm:Main}
to obtain that the Gigli-Mantegazza
flow for $\met_t$ is given,  for $\mg_t$ admissible
points, by the smooth metric $\mg_t$. 
\end{proof}

Next, let us consider the Witch's hat sphere.
We follow the Example 3.2 from \cite{Lakzian-Sormani}
Let $\phi: [0,\pi] \to [0,2]$ be a smooth cut-off function with
$$
	\phi(r) = 0,  \;\text{for}\; r \in \bbrac{0,\frac{\pi}{4}} \;\; \text{and} \;\; 
	\phi(r) = 1, \;\text{for}\; r \in \bbrac{\frac{3\pi}{4} , \pi}
$$
and such that
$$
	|\phi' (r)| \leq 1/10
$$
Let
$$
 	f(r) := \phi(r) \left(\frac{\pi - r}{\pi} \right) + (1-\phi(r))\sin(r) 
$$

Now take the metric $\mg_{\witch} = dr^2 + f(r)^2 \mg_{\Sph^n}$. 

 The identity map $\mathrm{Id}: (\Sph^{n+1}, \mg_{\Sph^{n+1}}) \to (\Sph^{n+1}, \mg_{\witch})$ is bi-Lipschitz as a map between two metric spaces and possesses a geometric conical singularity at one pole. 

\begin{proof}[Proof of Corollary \ref{Cor:Witch}]
The cone is obtained by gluing the following pieces via isometry between their boundaries.
Let
$$
	A_1= \bbrac{0,\frac{\pi}{4}} \times_f  \Sph^{n}, \;
	A_2= \bbrac{\frac{\pi}{4} , \frac{3\pi}{4}}  \times_f \Sph^n \; \text{and}\; 
	A_3= \bbrac{\frac{3\pi}{4}, \pi} \times_f \Sph^n.
$$
Then, $A_1$ is a spherical cap with constant sectional curvature equal to $1$. Hence,  it obviously is an Alexandrov space. Furthermore, $A_2$ is a smooth warped product with bounded sectional curvature and therefore it is also  Alexandrov. Lastly,   $A_3$ is the standard cone with cross sectional diameter $< \pi$ which is known to be Alexandrov by \cite{BBI}. 
So, by  Proposition \ref{prop:petrunin-gluing}, we obtain that it is an $\RCD(K,N)$ space.

Moreover, since the metric $\mg_{\witch}$ has a 
geometric conical singularity at one point,
and it is smooth away from that 
point, by Theorem \ref{Thm:Main}, 
we obtain that the Gigli-Mantegazza flow 
$\met_t$ is induced everywhere 
by the metric $\mg_t$, which is smooth everywhere
but at the singular point. 
\end{proof}

\bibliographystyle{amsplain}
\providecommand{\bysame}{\leavevmode\hbox to3em{\hrulefill}\thinspace}
\providecommand{\MR}{\relax\ifhmode\unskip\space\fi MR }
\providecommand{\MRhref}[2]{%
  \href{http://www.ams.org/mathscinet-getitem?mr=#1}{#2}
}
\providecommand{\href}[2]{#2}

\setlength{\parskip}{0mm}

\end{document}